\documentclass{article}

\usepackage[english]{babel}

\usepackage[letterpaper,top=2cm,bottom=2cm,left=3cm,right=3cm,marginparwidth=1.75cm]{geometry}

\usepackage{amsmath,amssymb,amsthm}
\usepackage{graphicx}
\usepackage{tikz}
\usepackage{enumerate}
\usepackage{circuitikz}
\usetikzlibrary{arrows,positioning}
\usepackage{caption}
\usepackage{subcaption}
\usepackage{xcolor}
\usepackage[colorlinks=true, allcolors=blue]{hyperref}

\newcommand{\ddt}{\frac{\rm d}{{\rm d}t}}
\newcommand{\ddts}{\tfrac{\rm d}{{\rm d}t}}

\newcommand{\dtau}{{\rm d}\tau}

\newcommand{\C}{\mathbb{C}}
\newcommand{\Cpo}{\mathbb{C}_\omega^+}
\newcommand{\R}{\mathbb{R}}

\newcommand{\N}{\mathbb{N}}

\newcommand{\Vc}{\mathcal{V}}
\newcommand{\Wc}{\mathcal{W}}
\newcommand{\Jc}{\mathcal{J}}

\newcommand{\ran}{{\rm ran}\,}
\newcommand{\dom}{{\rm dom}\,}

\newcommand{\mul}{{\rm mul}\,}

\newcommand{\re}{{\rm Re}\,}

\newcommand{\setdef}[2]{\left\{\ #1\ \left|\ \vphantom{#1} #2\ \right.\right\}}
\newcommand{\scpr}[2]{\left\langle #1,#2\right\rangle}

\newenvironment{smallpmatrix}
{\left(\begin{smallmatrix}}
{\end{smallmatrix}\right)}

\newenvironment{smallbmatrix}
{\left[\begin{smallmatrix}}
{\end{smallmatrix}\right]}

\usepackage{tikz-cd}
\usetikzlibrary{positioning,arrows,decorations.markings}
\tikzset{
block/.style={
  draw, 
  rectangle, 
  minimum height=1.5cm, 
  minimum width=2.5cm, align=center,
  fill=blue!20
  }, 
line/.style={->,>=latex'}
}
\tikzset{negated/.style={
      decoration={markings,
           mark= at position 0.5 with {
               \node[transform shape] (tempnode) {$\backslash\!\!\backslash$};
            }
       },
       postaction={decorate}
    }
}

\newtheorem{definition}{Definition}[section]
\newtheorem{theorem}[definition]{Theorem}
\newtheorem{proposition}[definition]{Proposition}
\newtheorem{lemma}[definition]{Lemma}
\newtheorem{corollary}[definition]{Corollary}
\newtheorem{remark}[definition]{Remark}

\theoremstyle{definition}

\newtheorem{example}[definition]{Example}

\title{A pseudo-resolvent approach to abstract differential-algebraic equations}

\author{Hannes Gernandt \thanks{Bergische Universit\"{a}t Wuppertal,
Gau\ss stra\ss e 20,   D-42119 Wuppertal, Germany  (\texttt{hannes.gernandt@ieg.fraunhofer.de})} \and Timo Reis \thanks{Technische Universit\"{a}t Ilmenau,  Weimarer Stra\ss e 25,  98693 Ilmenau, Germany (\texttt{timo.reis@tu-ilmenau.de}).} }

\begin{document}
\maketitle

\begin{abstract}
We study linear abstract differential-algebraic equations (ADAEs), and we introduce an index concept which is based on polynomial growth of a~pseudo-resolvent. Our approach to solvability analysis is based on degenerate semigroups. 
We apply our results to some examples such as distributed circuit elements, and a system obtained by heat-wave coupling.
\end{abstract}

\section{Introduction}
We study abstract differential-algebraic equations (ADAEs)
\begin{align}
\label{adae}
\ddts Ex(t)=Ax(t)+f(t),\quad  Ex(0)=Ex_0,
\end{align}
where $x_0\in X$, $f:[0,t_f]\to Z$, $E:X\rightarrow Z$, $A:X\supset \dom A\rightarrow Z$ are linear, and $X$, $Z$ are Banach spaces. In most of our results, we will additionally assume that $E$ is bounded, $A$ is densely defined and closed, and
$f\in L^1([0,t_f];Z)$.  The main difficulty of the equations of the form \eqref{adae} come from the possible non-invertibility of $E$. Systems of the form \eqref{adae} may occur in two different settings:
\begin{enumerate}[a)]
\item \emph{Finite dimensional differential-algebraic systems which are coupled with partial differential equations}. Such a~coupling typically emerges from boundary action at the involved PDEs. In this case, the state space is typically a~Cartesian product of $\mathbb{R}^n$ with a function space, the operators $E$ and $A$ in \eqref{adae} are block operator matrices.     Typical examples are electrical circuits which further contain semiconductor devices or parasitic effects like crosstalk or heat propagation \cite{Tis03}. Another example is formed by multibody systems coupled with components from continuum mechanics \cite{Sime13}.
\item \emph{Systems governed by degenerate PDEs}. These occur in spatial and time-dependent processes which have the property that there are spatial subdomains in which the dynamics is restricted or even not present. In this type $E$ is usually a~multiplication operator on a~function space, whereas $A$ is a~spatial differential operator. An example is the vibrating string in which the spatially dependent string density and/or tension vanish in a~part. Further, {\em magnetoquasistatic approximations} of the electromagnetic field can be also described by infinite dimensional systems of type \eqref{adae}, see e.g.\ \cite{BdGCS18}. 
\end{enumerate}
Our main assumption is that the ADAE \eqref{adae} is \textit{regular}, i.e.\  we have a non-empty \textit{resolvent set}
\begin{align*}
\rho(E,A):=\{\lambda\in\C ~|~ \text{$\lambda E-A:\dom A\to Z$ is bijective}\}.
\end{align*}
If $\lambda\in\rho(E,A)$, then the closed graph theorem \cite[Thm.~7.9]{Alt16} implies that the operator $(\lambda E-A)^{-1}:Z\to X$ exists and that it is bounded.\\
The degenerate Cauchy problems of the above form were studied in numerous works. We mention the two monographs by {\sc Favini} and {\sc Yagi} \cite{FaviYagi99} and {\sc Showalter} \cite{Show74}. Both follow the ansatz via multi-valued linear operators. In \cite{FaviYagi99}, a~classical solution concept is considered, which requires additional smoothness assumptions on $f$, whereas in \cite{Show74}, ADAEs with certain monotonicity assumptions on the involved operators $E$ and $A$ are considered.
Motivated by the approach to finite dimensional DAEs via the {\em tractability index} \cite{LamoMarz13}, ADAEs on Hilbert spaces are considered in \cite{ReisTisc05,Reis07} under the additional assumption that a~certain projector chain is existent and stagnant. A~normal form has been constructed by means of this projector chain, and this allows to derive a~wide solution theory and in particular the determination of {\em consistent initial values}, i.e., the set of $x_0\in X$ such that \eqref{adae} with given $f$ possesses a~solution.
There are also recent works \cite{TrosWaur18,TrosWaur19}, where ADAEs are studied, and a maximal space of initial values such that the solution of the degenerate Cauchy problem exists and is unique. The construction of this space of initial values is based on a previous construction from \cite{BergIlch12a} for finite dimensional $X=Z$, using so-called {\em Wong sequences}. It is however assumed that $X=Z$ are Hilbert spaces and that $E$ and $A$ are bounded, which particularly implies that the assumptions in \cite{ReisTisc05,Reis07} are fulfilled in this situation.
Furthermore, certain closedness assumptions are made on the spaces in the Wong sequences, which exclude nearly all practically relevant examples.
The approach via Wong sequences has been further used in \cite{Tros19}
for the analysis of ADAEs with operators $E$ and $A$ as in our situation, however with the additional assumption that $f$ vanishes.
Previously, ADAEs of the form $\frac{d}{dt}Ex(t)=Ax(t)$ were also considered in \cite{BasaCher02}, where it is assumed that the domains of powers of a linear relation become stationary after a finite number of steps. In the case where $X=Z$ and $E={\rm id}$, this implies that $A$ is indeed bounded. Hence, this assumption excludes many practical examples arising from partial differential equations.\\
Besides $X,Z$ being Hilbert spaces, a~joint assumption \cite{ReisTisc05,Reis07,TrosWaur18,TrosWaur19,Tros19} is that the resolvent is \emph{polynomially bounded} on some half-plane. That is, for some $\omega\in\R$ holds $\Cpo:=\{\lambda\in\C : \re \lambda>\omega\}\subseteq\rho(E,A)$, and there exists some $k\in\N$ and $M>0$ with
\begin{equation}
\label{newdefindex}
\|(\lambda E-A)^{-1}\|\leq M|\lambda|^{k-1},\quad \text{for all $\lambda\in\Cpo$.}
\end{equation}
If $E$ and $A$ are matrices, then the resolvent is rational, whence \eqref{newdefindex} is always fulfilled in the finite dimensional case as long as the resolvent $\rho(E,A)$ is non-empty.
It further follows from the Weierstra\ss\ canonical form, see e.g.\ \cite{BergIlch12a}, that the smallest $k$ such that \eqref{newdefindex} holds for some $\omega\in\R$, equals the so-called {\em (Kronecker) index of the DAE} \cite{LamoMarz13,KunkMehr06}. However, in contrast to the finite dimensional case, the polynomial growth rate does not coincide with the tractability index of ADAEs \cite{ReisTisc05}.

In this note, we present an approach to the solution theory of ADAEs based on  \textit{pseudo-resolvents} \cite{Hill48}, which are holomorphic operator-valued functions that satisfy a resolvent identity. Using a growth assumption similar to \eqref{newdefindex}, we derive an operator $A_R$ associated with the pseudo-resolvent. This operator can be used to obtain an {ordinary differential equation}
hidden in the ADAE \eqref{adae}. Under additional assumptions it can be shown that the operator that describes the ordinary differential equations generates a strongly continuous semigroup.

We develop a novel solution theory, which makes use of the concepts of {\em linear relations} (or also called {\em multi-valued operators}), and {\em degenerate semigroups}, where the latter is a~generalization of the well-known concept of strongly continuous semigroup.
Based on these results, we will derive a generalization of the variations of constants formula, which gives rise to mild solutions of the ADAE \eqref{adae}. We will further present relations to weak and classical solutions of ADAEs. In doing so, we neither have to impose any monotonicity assumptions on $E$ and $A$ nor existence and stagnancy of certain subspace iterations or projector chains.

Our solution theory mainly uses the assumption $\mathbf{(D_k)}$ for some positive integer $k$, which is a dissipativity condition for the pseudo-resolvent restricted to a certain subspace 
For finite dimensional systems,  $E,A\in\mathbb{C}^{n\times n}$ the condition $\mathbf{(D_k)}$ is equivalent to the DAE \eqref{adae} having nilpotency index at most $k$ and in this sense the develop solution theory applicable to  ADAEs of higher index. Furthermore, it is shown that under certain dissipativity and nonnegativity assumptions on the operators $E$ and $A$, we can guarantee that $\mathbf{(D_1)}$ or $\mathbf{(D_2)}$ is fulfilled. This is in alignment with recent results for semi-dissipative Hamiltonian DAEs, i.e.\ $E$ is nonegative and $A$ is dissipative, where it is shown that their nilpotency index is at most two \cite{MehlMehrWojt18}. This class of DAEs is a subclass of \emph{port-Hamiltonian} DAEs which are based on an unifying energy-based system formulation that has received a lot of attention in recent years, see \cite{MehU23}. 

Besides our approach there are many results on solutions theory for ADAEs available which will be briefly discussed in the following. In \cite{SvirFedo03} the authors already use pseudo-resolvent for the solution theory, but their solution theory is based on the assumption of so called weak $(E,p)$-radiality, which is also a polynomial growth bound on the pseudo-resolvent which appears to be slightly stronger then our condition $\mathbf{(D_k)}$. Furthermore, no conditions on the coefficients ensuring weak $(E,p)$-radiality were obtained. 

In \cite{TrosWaur18,TrosWaur19} the authors generalize the Wong sequence approach to solve DAEs from \cite{BergIlch12a} to ADAEs with bounded coefficients in Hilbert spaces using (among other assumptions) considered a polynomial growth condition on the resolvent $s\mapsto (sE-A)^{-1}$.

In \cite{JacoMorr22} the authors studied dissipative DAEs in the sense that a certain dissipativity condition holds and develop a solution theory for the homogeneous case. However their assumptions imply that our condition  $\mathbf{(D_1)}$ holds. Here we give also dissipativity based conditions under which $\mathbf{(D_2)}$ holds and in that sense our results can be seen as generalization of \cite{JacoMorr22} to higher index.

More recently, \cite{MehrZwar23} also considered dissipative ADAEs in a slightly more general framework compared to \cite{JacoMorr22} allowing for a Hamiltonian weight operator that is inspired by port-Hamiltonian systems. Compared to this, we focus here on developing a more general framework for the solution of ADAEs. 

In \cite{FaviYagi99} also pseudo-resolvents are considered in Banach spaces, but their assumptions to guarantee existence of classical solutions are stronger and no mild or weak solutions are considered.

The paper is organized as follows: In Section~\ref{sec:solconcept} we consider solution concepts for ADAEs, and we briefly present the main concepts of this article, namely linear relations, pseudo-resolvents and degenerate semigroups. The interrelation between these concepts is - in conjunction with ADAEs - be elaborated in Section~\ref{sec:mainconcept}. A brief discussion of different index concepts for regular  ADAEs~\eqref{adae} can be found in Section~\ref{sec:indexsol}. In Section~\ref{sec:diss} we consider a~special class of ADAEs in which the operators fulfill a~certain dissipativity conditions, and in Section~\ref{sec:pseudoDGL}, we study solvability of ADAEs via pseudo-resolvents. Finally, in Section \ref{sec:appl} we provide several examples where abstract differential-algebraic equations arise.

\textbf{Nomenclature}. Throughout this article, $\N$ is the set of natural numbers including zero. For Banach spaces $X$ and $Y$, $L(X,Y)$ stands for the space of bounded operators mapping from $X$ to $Y$, and we further abbreviate $L(X):=L(X,X)$.
The norm in a~space {$X$} will be denoted by $\|\cdot\|_{{X}}$ or {simply} $\|\cdot\|$, if clear from context. The restriction of $A:X\to Y$ to $Z\subset X$ is denoted by $A|_Z$. Further, for the dual $X'$ of $X$, $\scpr{x'}{x}_{X',X}$ stands for the evaluation of $x'\in X'$ at $x\in X$. 

The domain of an operator $A$ is denoted by $\dom A$. The identity operator on the space $X$ is $I_X$, or just $I$, if clear from context. The symbol $A'$ stands for the dual of a~linear operator $A$.

For an interval $I\subset \R$ and a~Banach space $X$, the space of $k$-times differentiable $X$-valued functions on $I$  will be denoted by $C^k(I;X)$ and the space of continuous functions by $C^0(I;X)=C(I;X)$. Further, for $p\in [1,\infty]$, $k\in\N$, the Lebesgue space $L^p$ and Sobolev space $W^{k,p}$ of $X$-valued functions on $I$ are denoted by
$W^{k,p}(I;X)$ and $L^p(I;X)$, where we further set $W^{k,p}(J):=W^{k,p}(I;\C)$ and $L^p(I):=L^{p}(I;\C)$. For a~subinterval $J\subset I$ we canonically identify $C^k(J;X)\subset C^k(I;X)$, $W^{k,p}(J;X)\subset W^{k,p}(I;X)$ and $L^p(J;X)\subset L^p(I;X)$ via restriction of functions to $J$.
 Further, $W^{k,p}_{\rm loc}(J;X)$ and $L^p_{\rm loc}(J;X)$ are the spaces of functions $f:I\to X$ with $f\in L^p(K;X)$ ($f\in W^{k,p}(K;X)$) for all compact subintervals $K\subset I$.

Note that, throughout this article, integration of $X$-valued functions always has to be understood in the Bochner sense \cite{Diestel77},  and Banach spaces are always assumed to be complex. 

\section{Solutions concepts for ADAEs }
\label{sec:solconcept}
The following solution concepts for ADAEs of type \eqref{adae} are considered in this article. 

\begin{definition}[Solution concepts]
\label{def:solution}
Let $X$, $Z$ be normed spaces, let $E:X\to Z$ and $A:X\supset\dom A\to Z$ be linear, let $t_f>0$, $f:[0,t_f]\to Z$ and $x_0\in X$.
\begin{itemize}
    \item[\rm (a)] $x:[0,t_f]\to X$ is called \textit{classical solution} of \eqref{adae}, if  $Ex(\cdot)\in C([0,t_f];Z)\cap C^1((0,t_f);Z)$, $(Ex)(0)=Ex_0$, $x(t)\in\dom A$ for all $t\in (0,t_f)$, and it holds $\ddts Ex(t)=Ax(t)+f(t)$ for all $t\in (0,t_f)$.
    \item[\rm (b)] $x:[0,t_f]\to X$ is called a \textit{mild solution} of \eqref{adae}, if $f\in L^1([0,t_f];Z)$, $Ex(\cdot)\in C([0,t_f];Z)$, $x\in L^1([0,t_f];X)$, and for almost all $t\in [0,t_f]$, it holds $\int_0^t x(\tau)\dtau\in \dom A$ with
    \begin{align}
    \label{eq:mildsol}
    Ex(t)-Ex_0=A\int_0^tx(\tau)\dtau +\int_0^tf(\tau)\dtau.
    \end{align}
    \item[\rm (c)] $x:[0,t_f]\to X$ is called \textit{weak solution} of \eqref{adae}, if $A$ is densely defined, $Ex(\cdot)\in C([0,t_f];Z)$ with $(Ex)(0)=Ex_0$, and for all $z'\in Z'$, it holds that $t\mapsto\langle f(t),z'\rangle_{Z,Z'}\in L^{1}([0,t_f])$, $t\mapsto\langle Ex(t),z'\rangle_{Z,Z'}\in W^{1,1}([0,t_f])$ with weak derivative fulfilling
\begin{equation}\ddts\langle Ex(t),z'\rangle=\langle x(t),A'z'\rangle+\langle f(t),z'\rangle    \label{eq:weaksol}\end{equation}
for almost all $t\in[0,t_f]$ and all $z'\in \dom A'$.
\end{itemize}
We call $x_0\in X$ a {\em classical (mild, weak) consistent initial value} for $\ddts Ex(t)=Ax(t)+f(t)$, if \eqref{adae} has a~classical (mild, weak) solution.
\end{definition}
Note that the notion of weak solutions has been used in \cite{ReisTisc05}, whereas, in the articles \cite{TrosWaur19,TrosWaur18}, mild solutions are studied. The definition of weak solution involves the dual of $A$, which is given by the operator $A':Z'\supset \dom A'\to X'$ which is defined on the set of bounded linear functionals $z'\in Z'$ for which the functional $\dom A\to\C:x\mapsto \langle Ax,z'\rangle$ has a~bounded extension to $X$, that is,
\[\dom A'=\{z'\in Z'\,:\,\exists c>0\text{ s.t.\ }|\langle Ax,z'\rangle|\leq c\,\|x\|\,\forall x\in X\}.\]
If $A$ is densely defined, then $A'$ is well-defined by
\[\langle x, A'z'\rangle_{X,X'}=\langle Ax,z\rangle_{Z,Z'}\;\forall\, x\in\dom A,\,z'\in\dom A'.\]
\begin{remark}
The definition of classical, mild and weak solutions can be easily extended to the infinite time horizon $[0,\infty)$. Namely, we say that $x:[0,\infty)\to X$ is a~classical (mild, weak) solution of \eqref{adae}, if, for all $t_f>0$, the restriction of $x$ to $[0,t_f]$ is a~classical (mild, weak) solution of \eqref{adae}. For mild solutions, we have to impose that
$f\in L^{1}_{\rm loc}([0,\infty);X)$, whereas for weak solutions $f:[0,\infty)\to X$ has to fulfill that $t\mapsto \langle f(t),z'\rangle\in L^{1}_{\rm loc}([0,\infty))$ for all $z'\in Z'$.
\end{remark}
Next we show the relation between these solution concepts.

\begin{theorem}
Let $X$, $Z$ be Banach spaces, let $E\in L(X,Z)$ and $A:X\supset\dom A\to Z$ be linear and closed. Further, let $t_f>0$, $x_0\in X$ and $f:[0,t_f]\to Z$. Then the following holds:
\begin{enumerate}[\rm (a)]
\item If $f\in L^1([0,t_f];Z)$, and $x:[0,t_f]\to X$ is a classical solution of \eqref{adae}, then $x$ is a mild solution of \eqref{adae}.
\item If $t\mapsto \langle f(t),z'\rangle \in L^1([0,t_f])$ for all $z'\in Z'$, $A$ is densely defined, and $x:[0,t_f]\to X$ is a classical solution of \eqref{adae}, then $x$ is a~weak solution.
\item If $A$ is densely defined, and $x:[0,t_f]\to X$ is a mild solution of \eqref{adae}, then $x$ a weak solution of \eqref{adae}.
\item If, $X$ and $Z$ are reflexive, $A$ is densely defined, $f\in L^1([0,t_f];Z)$ and $x:[0,t_f]\to X$ is a weak solution of \eqref{adae}, then $x$ is a~mild solution of \eqref{adae}.
\end{enumerate}
\end{theorem}
\begin{proof}
\begin{enumerate}[(a)]
\item  As a~consequence of Pettis's measurability theorem \cite[Thm.~II.1.2]{Diestel77} and the continuity of $Ex:[0,t_f]\to Z$, we have that $Ex$ is Bochner integrable. Now an integration of \eqref{adae} of $[0,t]$ for $0\leq t\leq t_f$ yields
\[Ex(t)-Ex(0)=\int_0^tAx(\tau)\dtau+\int_0^tf(\tau)\dtau.\]
Since, by \cite[Prop.~1.1.7.]{ArenBatt11}, we have $\int_0^tx(\tau)\dtau\in \dom A$ with
\[A\int_0^tx(\tau)\dtau=\int_0^tAx(\tau)\dtau,\]
we immediately obtain that $x$ is a~mild solution.
\item Assume that $t\mapsto \langle f(t),z'\rangle \in L^1([0,t_f])$ for all $z'\in Z'$, $A$ is densely defined, and $x:[0,t_f]\to X$ is a classical solution, and let $z'\in \dom A'$. Now applying $z'$ to $\ddts Ex=Ax+f$ and using the definition of $A'$, we obtain \eqref{eq:weaksol}. In particular, $x$ is a~weak solution.
\item Assume that $A$ is densely defined, $x:[0,t_f]\to X$ is a mild solution of \eqref{adae}, and let $z'\in \dom A'$. Then, by applying $z'$ to \eqref{eq:mildsol}, we obtain
\begin{multline*}\langle Ex(t),z'\rangle_{Z,Z'}-\langle Ex_0,z'\rangle_{Z,Z'}\\=\left\langle A\int_0^tx(\tau)\dtau,z'\right\rangle_{Z,Z'}+\left\langle\int_0^tf(\tau)\dtau,z'\right\rangle_{Z,Z'}.
\end{multline*}
Hence, the weak form of the fundamental theorem of calculus \cite[Thm.~E3.6]{Alt16} implies that $\langle Ex(\cdot),z'\rangle_{Z,Z'}\in W^{1,1}_{\rm loc}([0,t_f])$ and that \eqref{eq:weaksol} holds on $[0,t_f]$. That is, $x$ is a~weak solution.\\
\item
Assume that $X$ and $Z$ are reflexive, $A$ is densely defined, $f\in L^1([0,t_f];Z)$ and $x:[0,t_f]\to X$ is a weak solution of \eqref{adae}.
By another use of the weak form of the fundamental theorem of calculus \cite[Thm.~E3.6]{Alt16}, we obtain from an integration of \eqref{eq:weaksol} on $[0,t]$ for $0\leq t\leq t_f$ that
\[\langle Ex(t),z'\rangle-\langle Ex(0),z'\rangle=\int_0^t\langle x(\tau),A'z'\rangle \dtau+\int_0^t\langle f(\tau),z'\rangle \dtau.\]
The property $f\in L^1([0,t_f];Z)$ gives
\[\int_0^t\langle f(\tau),z'\rangle d\tau=\left\langle\int_0^t f(\tau),z' d\tau\right\rangle.\]
Since $x:[0,t_f]\to X$ is continuous, we can infer from Pettis's measurability theorem \cite[Thm.~II.1.2]{Diestel77} that $x(\cdot)$ is Bochner integrable, and thus
\[\int_0^t\langle x(\tau),A'z'\rangle d\tau=\left\langle \int_0^tx(\tau)d\tau,A'z'\right\rangle,\]
which gives
\[\left\langle Ex(t)- Ex(0)-\int_0^tf(\tau) d\tau,z'\right\rangle=\left\langle \int_0^tx(\tau)d\tau,A'z'\right\rangle\quad\forall z'\in \dom A'.\]
Consequently, the functional $\dom A'\to\C$ with $z'\mapsto\langle \int_0^tx(\tau)d\tau,A'z'\rangle$ has a~bounded extension to $Z'$, whence $\int_0^tx(\tau)d\tau\in \dom A''$. Now the reflexivity of $X$ leads to $\int_0^tx(\tau)d\tau\in \dom A$ with
\begin{equation}\left\langle Ex(t)- Ex(0)-A\int_0^tx(\tau)d\tau-\int_0^tf(\tau) d\tau,z'\right\rangle=0\quad\forall z'\in \dom A'.\label{eq:ADAE_mildweak}\end{equation}
Further, by using that $X$ and $Z$ are reflexive and $A$ is densely defined, we obtain from \cite[Thm.~III.5.29]{Kato80} that $A'$ is densely defined as well. Consequently, \eqref{eq:ADAE_mildweak} implies \eqref{eq:mildsol}.
\end{enumerate}
 \end{proof}

\section{Main concepts}
\label{sec:mainconcept}
The first concept that will be of major importance is that of {\em linear relations}, which will be more intensively treated in Section~\ref{sec:rel_res}.
\begin{definition}[Linear relation]
Let $X$ be a~normed space. A {\em (closed) linear relation $L$ in $X$} is a~(closed) subspace $L$ of $X\times X$.
\end{definition}
Linear operators can be identified as linear relations via their graphs. Hence linear relations can be interpreted as multi-valued linear operators. In the context of ADAEs \eqref{adae}, the linear relations
\begin{align}\label{eq:Ldef}
\begin{split}
L_l&=\,\setdef{(y,z)\in Z\times Z}{\exists\, x\in \dom A\;\text{s.t.}\; y=Ex\,\wedge z=Ax},\\
L_r&=\,\setdef{(x,w)\in \dom A\times X}{ Ew=Ax},
\end{split}
\end{align}
will be used. Next, we show that classical and mild solutions can be expressed by means of the linear relations $L_r$ and $L_l$.
\begin{theorem}
Let $X$, $Z$ be normed spaces, let $E:X\to Z$ be linear and bounded, and let  $A:X\supset\dom A\to Z$ be linear, densely defined and closed, let $t_f>0$, $f:[0,t_f]\to Z$ and $x_0\in X$. Further, let $L_l\subset Z\times Z$ and be defined as in \eqref{eq:Ldef}.
\begin{enumerate}[\rm (a)]
\item Let $x:[0,t_f]\rightarrow X$ be a classical solution of the ADAE \eqref{adae}. Then
\begin{align*}
Ex(0)=Ex_0\text{ and }\left(Ex(t),\ddts Ex(t)-f(t)\right)\in L_l\;\;\forall  t\in[0,t_f].
\end{align*}
\item Conversely, if $z\in C^1((0,t_f);Z)\cap C([0,t_f];Z)$ with $z(0)=Ex_0$ fulfills
\begin{align}
\label{inLl}
\left(z(t),\ddts z(t)-f(t)\right)\in L_l\;\;\forall  t\in[0,t_f],
\end{align}
then there exists some $x:[0,t_f]\to\dom A$ with $z(t)=Ex(t)$ for all $t\in[0,t_f]$. This function $x$ is a~classical solution of \eqref{adae}.
\item Assume that the ADAE \eqref{adae} is regular, let $f\in L^1([0,t_f];Z)$ and $\lambda\in\rho(E,A)$. If $x:[0,t_f]\to X$ is a~mild solution of \eqref{adae}, then 
\begin{multline}
\left(\int_{0}^t x(\tau)d\tau-(\lambda E-A)^{-1}\int_0^t f(\tau)\dtau\right. ,\\\qquad\qquad \left. x(t)-x_0-\lambda (\lambda E-A)^{-1}\int_0^t f(\tau)\dtau\right)\in L_r
\text{ for a.a.\ }t\in[0,t_f].
\label{eq:mildsolrel}\end{multline}
\item  Assume that the ADAE \eqref{adae} is regular, $f\in L^1([0,t_f];Z)$ and $\lambda\in\rho(E,A)$. If $x\in L^1([0,t_f];X)$ fulfills \eqref{eq:mildsolrel},
then $x$ is a~mild solution of \eqref{adae}.
\end{enumerate}
\end{theorem}
\begin{proof}
\begin{enumerate}[(a)]
\item Let $x:[0,t_f]\rightarrow X$ be a classical solution of the ADAE \eqref{adae}. Then $Ex(0)=Ex_0$, and for all $t\in[0,t_f]$, we have
\[\left(Ex(t),\ddts Ex(t)-f(t)\right)=\left(Ex(t),Ax(t)\right)\in L_l.\]
\item Assume that $z\in C^1((0,t_f);Z)\cap C([0,t_f];Z)$ fulfills $z(0)=Ex_0$ and \eqref{inLl}. Using the definition of $L_l$, we obtain that for all $t\in[0,t_f]$, there exists some $x(t)\in\dom A$, such that
\[z(t)=Ex(t)\text{ and }\ddts z(t)-f(t)=Ax(t).\]
Then $Ex\in C^1((0,t_f);Z)\cap C([0,t_f];Z)$ fulfills $Ex(0)=Ex_0$ and $\ddts Ex(t)-f(t)=Ax(t)$, whence $x$ is a~classical solution of \eqref{adae}.
\item 
Let $f\in L^1([0,t_f];Z)$, $x_0\in X$, $\lambda\in\rho(E,A)$, and let $x:[0,t_f]\to X$ be a~mild solution of \eqref{adae}. Then, for almost all $t \in[0,t_f]$,
\[\begin{aligned}
&E\left(x(t)-x_0-\lambda (\lambda E-A)^{-1}\int_0^t f(\tau)\dtau\right)\\
=&A\int_0^tx(\tau)d\tau+\int_0^tf(\tau)\dtau-\lambda E(\lambda E-A)^{-1}\int_0^tf(\tau)\dtau\\
=&A\int_0^tx(\tau)d\tau+(\lambda E-A-\lambda E)(\lambda E-A)^{-1}\int_0^tf(\tau)\dtau\\
=&A\int_0^tx(\tau)d\tau-A(\lambda E-A)^{-1}\int_0^tf(\tau)\dtau\\
=&A\left(\int_0^tx(\tau)d\tau-(\lambda E-A)^{-1}\int_0^tf(\tau)\dtau\right),
\end{aligned}\]
which shows that \eqref{eq:mildsolrel} holds. 

\item 
Let $f\in L^1([0,t_f];Z)$, $x_0\in X$, $\lambda\in\rho(E,A)$ and assume that $x\in L^1([0,t_f];X)$ fulfills \eqref{eq:mildsolrel}.
By $(\lambda E-A)^{-1}\int_0^tf(\tau)\dtau\in \dom A$ for almost all $t\in[0,t_f]$, we have that $\int_0^tx(\tau)\dtau\in \dom A$ for almost all $t\in[0,t_f]$. Further, \eqref{eq:mildsolrel} leads to
\[\begin{aligned}
&Ex(t)-Ex_0\\=&\,E\left(x(t)-x_0-\lambda (\lambda E-A)^{-1}\int_0^t f(\tau)\dtau\right)+\lambda E(\lambda E-A)^{-1}\int_0^t f(\tau)\dtau\\
=&\,A\left(\int_{0}^t x(\tau)d\tau-(\lambda E-A)^{-1}\int_0^t f(\tau)\dtau\right)+\lambda E(\lambda E-A)^{-1}\int_0^t f(\tau)\dtau\\
=&\,A\int_{0}^t x(\tau)d\tau-A(\lambda E-A)^{-1}\int_0^t f(\tau)\dtau+\lambda E(\lambda E-A)^{-1}\int_0^t f(\tau)\dtau\\
=&\,A\int_{0}^t x(\tau)d\tau+\int_0^t f(\tau)\dtau,
\end{aligned}\]
which shows that $x$ is a mild solution of \eqref{adae}.

\end{enumerate}
\hfill \end{proof}

If $0\in\rho(A)$, then  \eqref{eq:mildsolrel} can be simplified to
\[
\left(\int_{0}^t x(\tau)d\tau+A^{-1}\int_0^t f(\tau)\dtau,x(t)-x_0\right)\in L_r
\text{ for a.a.\ }t\in[0,t_f].
\]

The following result indeed allows to restrict regular ADAEs to ones in which $A$ is invertible.
\begin{lemma}\label{lem:shift}
Let $X$, $Z$ be normed spaces, let $E:X\to Z$, $A:X\supset\dom A\to Z$ be linear, let $t_f>0$, $f:[0,t_f]\to Z$, $x_0\in X$. Let $\mu\in\C$ and define $f_\mu:[0,t_f]\to X$ with
$f_{\mu}(t):=e^{-\mu t}f(t)$ for all $t\in[0,t_f]$. Then the following holds:
\begin{enumerate}[\rm (a)]
\item $x:[0,t_f]\to X$ is a classical solution of \eqref{adae} if, and only if, $x_\mu:[0,t_f]\to X$ with
$x_{\mu}(t):=e^{-\mu t}x(t)$ is a classical solution of
\begin{align}
    \label{mu_dgl}
    \tfrac{\rm d}{{\rm d} t} Ex_{\mu}(t)=(A-\mu E)x_{\mu}(t)+f_{\mu}(t),\quad Ex_\mu(0)=Ex_0.
\end{align}
\item If, moreover, $E$ is bounded, $A$ is closed and $f\in L^1([0,t_f];Z)$, then 
$x:[0,t_f]\to X$ is a weak solution of \eqref{adae} if, and only if, $x_\mu:[0,t_f]\to X$ with
$x_{\mu}(t):=e^{-\mu t}x(t)$ is a weak solution of \eqref{mu_dgl}.
\item If, moreover, $E$ is bounded, $A$ is densely defined and closed, and $f\in L^1([0,t_f];Z)$, then 
$x:[0,t_f]\to X$ is a mild solution of \eqref{adae} if, and only if, $x_\mu:[0,t_f]\to X$ with
$x_{\mu}(t):=e^{-\mu t}x(t)$ is a mild solution of \eqref{mu_dgl}.
\end{enumerate}
\end{lemma}
\begin{proof} By reversing the roles of $x$ and $x_\mu$, it suffices to assume that $x$ being a classical (weak, mild) solution of \eqref{adae} implies that $x_\mu$ is a 
classical (weak, mild) solution of \eqref{mu_dgl}.\\
The statement for classical solutions follows immediately from the product rule $\ddt Ex_\mu(t)=-\mu Ex_\mu(t)+e^{-\mu t}\ddt Ex(t)$. Likewise, the result for weak solutions follows 
from the product rule for weak derivatives \cite[p.~124]{Alt16}, which gives \[\ddts\langle Ex_\mu(t),z'\rangle=-\mu \langle Ex_\mu(t),z'\rangle+e^{-\mu t}\ddts\langle Ex(t),z'\rangle\quad\forall z'\in Z'.\]
To prove the remaining statement, assume that $x$ is a~mild solution of \eqref{adae}. Then an integration by parts yields that for almost all $t\in[0,t_f]$,
\[\int_0^tx_\mu(\tau)\dtau
=\int_0^te^{-\mu t}x(\tau)\dtau
=\mu\int_0^te^{-\mu \tau}\int_0^\tau x(\sigma){\rm d}\sigma\dtau+e^{-\mu t}\int_0^tx(\tau){\rm d}\tau.
\]
Now using the closedness of $A$, we obtain that $\int_0^tx_\mu(\tau)\dtau\in\dom A$ for almost all $t\in[0,t_f]$, and
\begin{align}
\label{Arein}
A\int_0^tx_\mu(\tau)\dtau=\mu\int_0^te^{-\mu \tau}A\int_0^\tau x(\sigma){\rm d}\sigma\dtau+e^{-\mu t}A\int_0^tx(\tau){\rm d}\tau.
\end{align}
Invoking \eqref{eq:mildsol}, \eqref{Arein} and the following integration by parts
\[\int_0^tf_\mu(\tau)\dtau
=\mu\int_0^te^{-\mu \tau}\int_0^\tau f(\sigma){\rm d}\sigma\dtau+e^{-\mu t}\int_0^tf(\tau){\rm d}\tau,
\]
we obtain
\[\begin{aligned}
&(A-\mu E)\int_0^tx_\mu(\tau)\dtau+\int_0^tf_\mu(\tau)\dtau\\
=&\mu\int_0^te^{-\mu \tau}A\int_0^\tau x(\sigma){\rm d}\sigma\dtau+e^{-\mu t}A\int_0^tx(\tau){\rm d}\tau-\mu E\int_0^tx_\mu(\tau)\dtau+\int_0^tf_\mu(\tau)\dtau\\
=&\mu\int_0^te^{-\mu \tau}\left(Ex(\tau)-Ex_0-\int_0^\tau f(\sigma){\rm d}\sigma\right)\dtau\\&\qquad+e^{-\mu t}\left(Ex(t)-Ex_0-\int_0^tf(\tau)\dtau\right)-\mu E\int_0^tx_\mu(\tau)\dtau+\int_0^tf_\mu(\tau)\dtau\\
=&\mu\int_0^tEx_\mu(\tau)\dtau+(e^{-\mu \tau}-1)Ex_0-\mu\int_0^te^{-\mu\tau}\int_0^\tau f(\sigma){\rm d}\sigma\dtau\\&\qquad+e^{-\mu t}Ex(t)-e^{-\mu t}Ex_0-e^{-\mu t}\int_0^tf(\tau)\dtau-\mu E\int_0^tx_\mu(\tau)\dtau+\int_0^tf_\mu(\tau)\dtau\\
=&E x_\mu(t)-Ex_0-\mu\int_0^te^{-\mu\tau}\int_0^\tau f(\sigma){\rm d}\sigma\dtau-e^{-\mu t}\int_0^tf(\tau)\dtau+\int_0^tf_\mu(\tau)\dtau\\
=&E x_\mu(t)-Ex_0,
\end{aligned}\]
which yields that $x_\mu$ is a~mild solution of
\eqref{mu_dgl}.
\end{proof}
Another important concept in conjunction with regular ADAEs are  {\em pseudo-resolvents} in the sense of Hille \cite{Hill48}, see also \cite{Haas06,HillPhil96,Kato59}.
\begin{definition}[Pseudo-resolvent]\label{def:res}
Let $X$ be a Banach space and $\Omega\subset\C$ be open. A mapping $R:\Omega\rightarrow L(X)$ is called a \textit{pseudo-resolvent}, if the following {\em resolvent identity} holds
\begin{align}
\label{pseudoresid}
\frac{R(\lambda)-R(\mu)}{\mu-\lambda}=R(\lambda)R(\mu),\quad \text{for all $\lambda,\mu\in\Omega$, $\lambda\neq\mu$.}
\end{align}
\end{definition}
Note that pseudo-resolvents are holomorphic with $R'(\lambda)=-R(\lambda)^2$. Further, by using
\[R(\mu)=R(\lambda)(I-(\lambda-\mu)R(\lambda))^{-1},\]
we can use the Neumann series \cite[Sec.~5.7]{Alt16} to obtain that a~power series expansion of the resolvent at $\lambda_0\in\Omega$ is given by
\begin{equation}\label{eq:powser}
R(\mu)=\sum_{n=0}^\infty (\lambda_0-\lambda)^n R(\lambda_0)^{n+1}.
\end{equation}
Any regular ADAEs can be associated with the {\em left and right resolvents}
\begin{align}
\label{def:RrRl}
\begin{aligned}
R_l:&&\rho(E,A)\to&\, L(Z),&& \text{with}\quad R_l(\lambda)=E(A-\lambda E)^{-1},\\
R_r:&&\rho(E,A)\to&\, L(X)&& \text{with}\quad R_r(\lambda)=(A-\lambda E)^{-1}E,
\end{aligned}\end{align}
and it is straightforward to verify that these mappings are pseudo-resolvents in the sense of Definition~\ref{def:res}.\\
By using Lemma~\ref{lem:shift}, we see that, for $\mu\in\rho(E,A)$, the (classical, weak, mild) solutions $x$ of the ADAE \eqref{adae} are connected to those of the \textit{pseudo-resolvent differential equation}
\begin{align}
\label{psresDGL2_int}
\tfrac{\rm d}{{\rm d}t}R_r(\mu)x_{\mu}(t)=x_{\mu}(t)+g_{\mu}(t),
\end{align}
via
 $x_{\mu}(t)=e^{-\mu t}x(t)$ and $g_{\mu}(t)=e^{-\mu t}(A-\mu E)^{-1}f(t)$ for all $t\in I$.
Our solution theory for ADAEs presented in this article is indeed based on the solution of pseudo-resolvent differential equations~\eqref{psresDGL2_int} (or an analogous pseudo-resolvent differential equation for $R_l$).

We further use the close connection between pseudo-resolvents and \textit{degenerate semigroups}. 
\begin{definition}[Degenerate semigroup]
Let $X$ be a Banach space. A strongly continuous mapping $T:(0,\infty)\rightarrow L(X)$ that satisfies
\begin{align}
\label{degsg}
T(t+s)=T(t)T(s),\quad \text{for all $s,t>0$,}
\end{align}
and $\sup_{0< t\leq 1}\|T(t)\|<\infty$ is called \textit{degenerate semigroup}.
\end{definition}
We will oftentimes assume that $X$ is reflexive. In this case, it can be concluded from \cite[Cor.~2.2]{Aren01} that the strong limit $T(0):=\lim_{\tau\searrow\,0}T(\tau)$ exists and therefore $T(0)$ is bounded. The property \eqref{degsg} implies that $T(0)$ is a bounded projector onto $X_0:=\ran T(0)$. The restriction  $(T(t)|_{X_0})_{t\geq 0}$ is then a strongly continuous semigroup on $X_0$. Further, for $X_1=\ran(I-T(0))$, we have $T(t)|_{X_1}=0$ for all $t\geq 0$.
Further note that for any degenerate semigroup $(T(t))_{t\geq 0}$ there exist $\omega\in\R$, $M\geq 1$, such that
\[
\|T(t)\|\leq Me^{\omega t}\quad \forall\,t\geq 0,
\]
which we refer to as {\em $\omega$-stability} of $(T(t))_{t\geq 0}$. If $\omega<0$ we say that the semigroup is \textit{exponentially stable}. For an $\omega$-stable degenerate semigroup, the Laplace transform defines a pseudo-resolvent
\begin{align}
\label{lapalce_pseudo}
R(\lambda)z:=\int_0^\infty e^{-\lambda t}T(t)z{\rm d}t\quad\text{for all $\lambda\in\Cpo$,}
\end{align}
see  \cite[Prop.~2.2]{Aren87}.

\section{Interplay between the main concepts}
\label{sec:rel_res}
\subsection{Linear relations and pseudo-resolvents}
We recall some basic notions for \textit{linear relations} between Banach spaces $X$ and $Y$ which are subspaces of the Banach space $X\times Y$. For an overview on this topic, see \cite{Cross}. 
Linear relations can be viewed as multi-valued linear operators and in this sense one can introduce the well known notions of \textit{domain, kernel, range} and \textit{multi-valued part}
\begin{align*}
    \dom L&:=\{x\in X \,|\, (x,y)\in L\}, &
    \ker L&:=\{x\in X \,|\, (x,0)\in L\},\\
    \ran L&:=\{y\in Y \,|\, (x,y)\in L\}, &
    \mul L&:=\{y\in Y \,|\, (0,y)\in L\}.
\end{align*}
Linear operators can be identified with linear relations via their graphs. 

Moreover, one can define the \textit{inverse} of a linear relation
\[
L^{-1}:=\{(y,x)\in Y\times X \,|\, (x,y)\in L\}.
\]
In the following, we assume that $L$ is in $X\times X$ and consider for $\lambda\in\C$
\[
L-\lambda:=\{(x,y-\lambda x)\,|\, (x,y)\in L\}.
\]
We say that $\lambda\in\C$ is in the resolvent set $\rho(L)$ of a linear relation $L$, if $\ker(L-\lambda)=\{0\}$ and $\ran(L-\lambda)=X$. If $L$ is closed, then, by \cite[Thm.~1.6]{FaviYagi99}, $\rho(L)$ is open.
Moreover, the closedness of $L$ together with the closed graph theorem \cite[Sec.~7.9]{Alt16} implies that $(L-\lambda)^{-1}$ is the graph of a~bounded linear operator, which is, for convenience, also denoted by $(L-\lambda)^{-1}\in L(X)$. It is shown e.g.\ in \cite[Thm.~1.8]{FaviYagi99} that the function $\lambda\mapsto(L-\lambda)^{-1}\in L(X)$ fulfills the resolvent identity, i.e.,
\begin{align}
\label{LRresid}
(L-\lambda)^{-1}-(L-\mu)^{-1}=(\mu-\lambda)(L-\lambda)^{-1}(L-\mu)^{-1},\quad \lambda,\mu\in\rho(L),
\end{align}
whence $\lambda\mapsto(L-\lambda)^{-1}\in L(X)$ on $\Omega=\rho(L)$ is a~pseudo-resolvent.\\
Note that \eqref{LRresid} further implies that the subspaces $\ker (L-\lambda)^{-1}$ and $\ran (L-\lambda)^{-1}$ are independent of the particular choice of $\lambda\in\rho(L)$.

We have seen that a~closed linear relation defines a~pseudo-resolvent. In the following, we show that also the converse is true in some sense. That is, by starting with a~pseudo-resolvent $R$, we can construct a~linear relation $L$ with the property $R(\lambda)=(L-\lambda)^{-1}$, see also~\cite{BasaCher02}, \cite[Prop.~A.2.4]{Haas06}, \cite[Lem.~9.2]{Pazy83}.
\begin{proposition}
\label{prop:relationfrompsres}
Let $X$ be a~Banach space and $R:\Omega\rightarrow L(X)$ a pseudo-resolvent, then for $\lambda\in\Omega$ the linear relation
\[
L_\lambda=\ran \begin{pmatrix}R(\lambda)\\ I+\lambda R(\lambda)\end{pmatrix}
\]
is closed and independent of the particular choice of $\lambda$, that is, $L_\lambda=L_\mu$ for all $\lambda,\mu\in\Omega$ and
\begin{equation*}
\begin{aligned}
\ker R(\lambda)=&\,\mul L_{\lambda}=\mul L_{\mu}=\ker R(\mu),\\
\ran R(\lambda)=&\,\dom L_{\lambda}=\dom L_{\mu}=\ran R(\mu). \end{aligned}\end{equation*}
Moreover, it holds $(L_{\lambda}-\lambda)^{-1}=R(\lambda)$ for all $\lambda\in\Omega$.
\end{proposition}

\subsection{Pseudo-resolvents and degenerate semigroups}\label{sec:psdegen}
We have already seen that in a reflexive Banach space $X$, the Laplace transform of a degenerate semigroup leads to a pseudo-resolvent via \eqref{lapalce_pseudo}. In this part, we treat the reverse statement. That is, to a given a pseudo-resolvent, we will associate a degenerate semigroup.\\
We will often assume that the pseudo-resolvent $R:\Omega\rightarrow L(X)$ fulfills, for $k\geq1$,\\[2mm]
$\mathbf{(D_k)}$:\quad There exist $\omega\in\R$, $M>0$, such that 
$[\omega,\infty)\subseteq\Omega$ and \[\|R(\lambda)x\|\leq \frac{M}{\lambda-\omega}\,\|x\|\quad
\forall\,\lambda\in(\omega,\infty), x\in \ran R(\omega)^{k-1}.
\]
We will show in Theorem~\ref{thm:main} that this assumption leads to a~closed operator $A_R$ defined on  a~dense subspace of $\overline{\ran R(\omega)^k}$,  $\omega\in\Omega$, such that $R(\lambda)|_{\overline{\ran R(\omega)^k}}=(A_R-\lambda I)^{-1}$. Under some additional assumptions such as $M=1$ in $\mathbf{(D_k)}$, it can be shown that $A_R$ is the generator of a strongly continuous semigroup $(T(t))_{t\geq 0}$. A degenerate semigroup $(T_R(t))_{t\geq 0}$ is then given by
\begin{align*}
T_R(t)(x_1+x_2):= T(t)x_1,\quad  \forall~ x_1\in \overline{\ran R(\omega)^k},~x_2\in\ker R(\omega)^k,~t\geq 0.
\end{align*}

 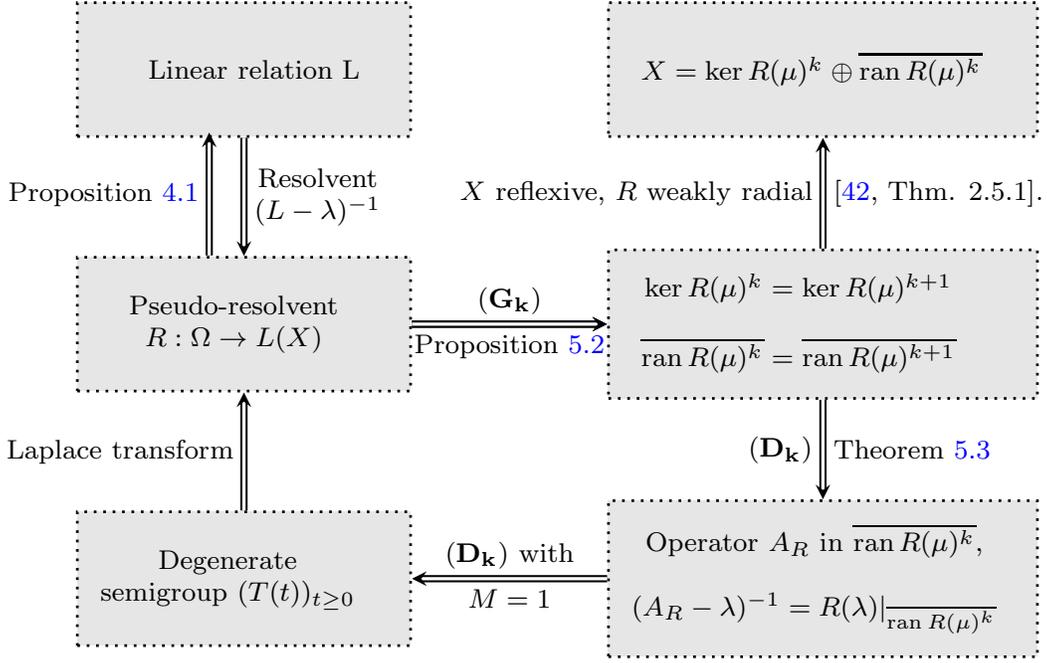
\begin{figure}
 {\footnotesize
 \begin{center}
   \resizebox{0.9\linewidth}{!}{\begin{tikzpicture}[thick,node distance = 22ex,implies/.style={double,double equal sign distance,-implies}, box/.style={fill=white,rectangle, draw=black},
   blackdot/.style={inner sep = 0, minimum size=3pt,shape=circle,fill,draw=black},plus/.style={fill=white,circle,inner sep = 0,very thick,draw},
   metabox/.style={inner sep = 2ex,rectangle,draw,dotted,fill=gray!20!white}]


     \node (LR)     [metabox,minimum size=1.4cm]  {\parbox{3cm}{\hspace{0cm}
     \text{~~~~Linear relation L}}};

     \node (PR)     [below of = LR, metabox,minimum size=1.4cm,yshift=0ex]  {\parbox{3cm}{\hspace{0cm}\text{~~~$\begin{matrix}\text{Pseudo-resolvent} \\  R:\Omega\rightarrow L(X)\end{matrix}$}}};

         \node (DS)     [below of = PR, metabox,minimum size=1.4cm,yshift=0ex]  {\parbox{3cm}{\hspace{0cm}$\begin{matrix} \text{Degenerate}\\ \text{semigroup $(T(t))_{t\geq 0}$}\end{matrix}$}};

\node (AR) [right of = DS, metabox, minimum size=1.4cm,xshift=28ex]  {\parbox{4cm}{\hspace{0cm}$\begin{matrix} \text{Operator $A_R$ in $\overline{\ran R(\mu)^k}$},\\  \\ (A_R-\lambda)^{-1}=R(\lambda)|_{\overline{\ran R(\mu)^k}}\end{matrix}$}};

\node (WD) [right of = LR, metabox, minimum size=1.4cm,xshift=28ex]  {\parbox{4cm}{\hspace{0cm} $X=\ker R(\mu)^k\oplus \overline{\ran R(\mu)^k}$}};

\node (DI) [right of = PR, metabox, minimum size=1.4cm,xshift=28ex]  {\parbox{4cm}{\hspace{0cm} $\begin{matrix} \ker R(\mu)^k=\ker R(\mu)^{k+1}\\ \\ \overline{\ran R(\mu)^k}=\overline{\ran R(\mu)^{k+1}}\end{matrix}$}};

 \node (LR2) [left of = LR, xshift = 19ex, yshift = -4.5ex] {};

 \node (PR2) [left of = PR, xshift = 19ex, yshift = 5ex] {};

 \draw[->,semithick,double,double equal sign distance,>=stealth] (LR)   -- node [right] {$\begin{matrix}\text{Resolvent}\\ (L-\lambda)^{-1}\end{matrix}$}  (PR);

 \draw[->,semithick,double, double equal sign distance,>=stealth] (PR2)   -- node [left] {Proposition~\ref{prop:relationfrompsres}}  (LR2);

 \draw[->,semithick,double, double equal sign distance,>=stealth] (DS)   -- node [left] {Laplace~transform}  (PR);

 \draw[->,semithick,double, double equal sign distance,>=stealth] (PR)   -- node [above] {$\mathbf{(G_k)}$}  (DI);

 \draw[->,semithick,double, double equal sign distance,>=stealth] (PR)   -- node [below] {Proposition~\ref{prop:ker_stat}}  (DI);

 \draw[->,semithick,double, double equal sign distance,>=stealth] (DI)   -- node [right] {Theorem~\ref{thm:main}}  node [left] {$\mathbf{(D_k)}$}  (AR);

  \draw[->,semithick,double, double equal sign distance,>=stealth] (AR)   -- node [below] {$M=1$}  node [above] {$\mathbf{(D_k)}$ with} (DS);


     \draw[->,semithick,double, double equal sign distance,>=stealth] (DI)   -- node [left] {$X$ reflexive, $R$ weakly radial}  node[right] {\cite[Thm. 2.5.1]{SvirFedo03}.} (WD);

 \end{tikzpicture}}
 \end{center}}
    \caption{The interplay between linear relations, pseudo-resolvents and degenerate semigroups in a Banach space $X$.}
    \label{fig:interplay}
  \end{figure}

\section{Decoupling of spaces and linear relations}\label{sec:declinrel}
Given a regular ADAE~\eqref{adae} and hence a pseudo-resolvent $R$, then additional assumptions have to be made to obtain a degenerate semigroup, see Figure~\ref{fig:interplay}. To this end, we establish certain space decompositions based on $R$. This will be used to decompose the associated linear relation and to construct a certain operator $A_R$ which is then utilized to solve the ADAE~\eqref{adae}.

In the following, we present a space decomposition which is inspired by the ideas on Wong sequences \cite{Camp82,BergIlch12a,TrosWaur18,TrosWaur19}, namely, for sufficiently large $k\in\N$,
\begin{align}
\label{finite_dec}
X=\overline{\ran R(\lambda)^k}\oplus \ker R(\lambda)^k.
\end{align}
If $X$ is finite dimensional this decomposition follows from the famous Weierstra\ss\ canonical form \cite{Gant59a}. The smallest $k$ such that \eqref{finite_dec} holds, is the index of the associated DAE, see \cite[p.~37]{Camp82} and \cite[Prop.~2.3]{BergIlch12a}. The subspaces $\mathcal{V}_k:=\ran R(\lambda)^k$ and $\mathcal{W}_k:=\ker R(\lambda)^k$ for $k\geq 0$ are independent from the choice of $\lambda$ and called {\em Wong sequences}, see \cite[Lem.~2.2]{BergIlch12a}. Wong sequences were generalized in \cite{TrosWaur18,TrosWaur19,Tros19} to Banach and Hilbert spaces but using geometric approach from \cite{BergIlch12a}, which imposes strong additional assumptions on the involved operators $E$ and $A$. 

To establish the space decomposition \eqref{finite_dec} in Banach spaces for $k=1$, a growth condition has been imposed in numerous works \cite{Aren01,BasaCher02,FaviYagi99,Kato59,ThalThal99}.\\[2mm]
$\mathbf{(G_1)}$:\quad There exists some $M>0$ and a sequence  $(\lambda_n)_{n\geq 1}$ in $\Omega\cap\R$ such that 
\[\text{$(\lambda_n)
\to\infty\,$  and }\, \|\lambda_n R(\lambda_n)\|\leq M\quad \forall\, n\geq 1.
\]
It has been pointed out by Kato in~\cite{Kato59} that this condition implies
\begin{align*}
\ker R(\lambda)\cap\overline{\ran R(\lambda)}=\{0\}\quad \forall\, \lambda\in\Omega,
\end{align*}
where it has been moreover shown that, under the additional assumption that $X$ \emph{locally sequentially weakly 
compact}, i.e.\ every bounded sequence has a weakly convergent 
subsequence,
\begin{align}
\label{k=1_dec}
X=\overline{\ran R(\lambda)}\oplus \ker R(\lambda)\quad \forall\, \lambda\in\Omega,
\end{align}
Note that, by \cite[Thm.~8.8]{Alt16}, this 
property is fulfilled, if $X$ is reflexive.

In the remainder, we show that a~decomposition \eqref{finite_dec} is possible, if the growth condition $\mathbf{(G_1)}$ is replaced by some weaker assumption.
To this end,we present a~result that, loosely speaking, kernels and ranges of products of pseudo-resolvents do not alter with the change of the resolvent points.
\begin{lemma}{\cite[Thm.~2.1.2]{SvirFedo03}}
Let $X$ be a Banach space and let $R:\Omega\rightarrow L(X)$ be a pseudo-resolvent and $\ell\in\N$,  $\lambda_i,\mu_i\in\Omega$ for $i=1,\ldots,\ell$. Then
\begin{align*}
\ker R(\lambda_1)\cdot\ldots\cdot R(\lambda_\ell)=\ker R(\mu_1)\cdot\ldots\cdot R(\mu_\ell), \\
\ran R(\lambda_1)\cdot\ldots\cdot R(\lambda_\ell)=\ran R(\mu_1)\cdot\ldots\cdot R(\mu_\ell). 
\end{align*}
\end{lemma}

Next we introduce a polynomial growth condition for the resolvent, which generalizes $\mathbf{\mathbf{(G_1)}}$. Namely, for $k\geq 1$,\\[2mm]
$\mathbf{(G_k)}$:\quad There exist $M>0$ and a sequence $(\lambda_n)_{n\geq 1}$ in $\Omega\cap\R$ such that \[  \text{$(\lambda_n)
\to\infty\,$ and }\, \|\lambda_n^{2-k} R(\lambda_n)\|\leq M\quad \forall\,n\geq 1.
\]
We first show that this growth condition implies that kernels and ranges of pseudo-resolvents stabilize at the $k$th power. We will make use of the dual pseudo-resolvent $R':\Omega\to L(X')$, which is given by
\begin{align}
\label{dual_def}
\langle R'(\lambda)x',x\rangle_{X',X}:=\langle x',R(\lambda)x\rangle_{X',X}\quad \forall~x\in X,\,x'\in X',\,\lambda\in\Omega.
\end{align}
It can be directly seen that $R'$ is a~pseudo-resolvent as well.

\begin{proposition}
\label{prop:ker_stat}
Let $X$ be a Banach space and let $R:\Omega\rightarrow L(X)$ be a pseudo-resolvent which satisfies $\mathbf{\mathbf{(G_{k_0})}}$ for some $k_0\geq 1$. Then
\begin{align}
\label{stat_ker_ran}
\ker R(\mu)^{k+1}=\ker R(\mu)^{k_0},\quad \text{and}\quad \overline{\ran R(\mu)^{k+1}}=\overline{\ran R(\mu)^{k_0}}
\end{align}
for some (and hence for all) $\mu\in\Omega$ and for all $k\geq k_0$.
\end{proposition}
\begin{proof}
Let $\mu\in\Omega$ and $k_0\geq 1$. First note that
\[R(\mu)\ker R(\mu)^{m}\subseteq \ker R(\mu)^{m-1}\subseteq \ker R(\mu)^{m}\quad \text{ for all $m\geq 2$.}
\]
This implies that $R_m:=R(\mu)|_{\ker R(\mu)^m}$ is nilpotent satisfying $R_m^m=0$ for all $m\geq 2$ and we  set $R_1:=0$. Hence, for all $m\geq 1$,
\begin{align}
\label{eq:nil_inv}
(I-(\mu-\lambda) R_m)^{-1}=\sum_{i=0}^{m-1}(\mu-\lambda)^i R_m^i.
\end{align}
Equation 
\eqref{eq:nil_inv} implies that for all $x\in\ker R(\mu)^{m}$, 
\begin{align}
\nonumber R(\lambda)x&=R_m(I-(\mu-\lambda)R_m)^{-1}x
\\&=R_m\sum_{i=0}^{m-1}(\mu-\lambda)^i R_m^ix\nonumber\\ &=\sum_{i=0}^{m-1}(\mu-\lambda)^iR^{i+1}_mx=\sum_{i=1}^{m-1}(\mu-\lambda)^{i-1}R^{i}_mx.
\label{nil_inv}
\end{align}
Since $\mathbf{\mathbf{(G_{k_0})}}$ holds for some $k_0\geq 1$, there exists some $M\geq 1$ and a~sequence $(\lambda_n)_n$ in $\Omega\cap\R$ with $\lambda_n\to\infty$ such that $\|\lambda_n^{2-k_0}R(\lambda_n)\|\leq M$. Assume that the opposite of \eqref{stat_ker_ran} holds. That is, there exists some $x_0\in \ker R(\mu)^{k_0+1}\setminus\ker R(\mu)^{k_0}$. Then, by $R(\mu)^{k_0}x_0\neq 0$ together with \eqref{nil_inv} for $m=k_0+1$, we are led to
\begin{align}
\|\lambda_n^{2-k_0}R(\lambda_n)x_0\|
&=\left\|\lambda_n^{2-k_0}\sum_{i=0}^{k_0}(\mu-\lambda_n)^{i-1}R^{i}_{k_0+1}x_0\right\|. \label{G_k_estim}
\end{align}
Since $R_{k_0+1}^{k_0}x_0\neq 0$, the left hand side in \eqref{G_k_estim}, is not uniformly bounded in $n$. This contradicts the growth assumption $\mathbf{(G_{k_0})}$. Therefore $\ker R(\mu)^{k_0+1}=\ker R(\mu)^{k_0}$ holds. By repeating the above arguments for all $k>k_0$ we conclude $\ker R(\mu)^{k+1}=\ker R(\mu)^{k_0}$.

To prove the second equation in \eqref{stat_ker_ran},  we consider the dual pseudo-resolvent $R'$ given by \eqref{dual_def}. Since the norm of a bounded operator and its dual coincide \cite[Thm.~4.10]{Rudi91}, $R'$ satisfies $\mathbf{(G_{k_0})}$ as well, and \eqref{stat_ker_ran} yields $\ker R'(\mu)^{k+1}=\ker R'(\mu)^{k_0}$ for all $k\geq k_0$. Define the {\em annihilator} of a~subspace $V\subseteq X'$ by 
\[
V_{\perp}=\setdef{x\in X}{\langle x,x'\rangle=0\;\forall\, x'\in V}.
\]  
Then \cite[Thm.~4.12]{Rudi91} leads to
\[
\overline{\ran R(\mu)^{k_0}}=(\ker R'(\mu)^{k_0})_{\perp}=(\ker R'(\mu)^{k})_{\perp}=\overline{\ran R(\mu)^{k}}.
\] 
\end{proof}

In the following, we impose the assumption ${\mathbf{(D_k)}}$ to obtain a space decoupling and an operator $A_R$ associated to a pseudo-resolvent as shown in Figure~\ref{fig:interplay}.

As one of the main results of this article, we show that $\mathbf{(D_k)}$ implies stationarity of power of ranges and kernels and we construct an operator $A_R$ as in Figure~\ref{fig:interplay}, which can be used later to solve inhomogeneous ADAEs. 
\begin{theorem}
\label{thm:main}
Let $X$ be a Banach space and let $R:\Omega\rightarrow L(X)$ be a pseudo-resolvent satisfying $\mathbf{(D_k)}$ for some $k\geq 1$ and let $\mu\in\Omega$. Then the following holds.
\begin{itemize}
    \item[\rm (a)] $R$ satisfies $\mathbf{(G_k)}$;
    \item[\rm (b)] $\overline{\ran R(\mu)^{k+1}}=\overline{\ran R(\mu)^k}$,
    $\ker R(\mu)^{k+1}=\ker R(\mu)^k$, and
    \begin{align}
    \label{empty_int}
    \overline{\ran R(\mu)^k}\cap \ker R(\mu)^k=\{0\};
    \end{align}
    \item[\rm (c)] The operator $A_{R}$ given by
\begin{align}
\label{ahat}
{\rm gr\,} A_{R}:=\left\{(R(\mu)x,x+\mu R(\mu)x) ~|~ x\in \overline{\ran R(\mu)^{k}}\right\}\subseteq \overline{\ran R(\mu)^k}\times \overline{\ran R(\mu)^k}
\end{align}
does not depend on the choice of $\mu\in\Omega$, is closed and densely defined on $\overline{\ran R(\mu)^k}$ with $\Omega\subseteq\rho(A_R)$ and for all $\lambda\in\Omega$ we have
\begin{align*}
(A_R-\lambda I)^{-1}=R(\lambda)|_{\overline{\ran R(\lambda)^{k}}}.
\end{align*}
\end{itemize}
\end{theorem}
\begin{proof}
To prove (a), we first observe that $\mathbf{(D_1)}$ implies $\mathbf{(G_1)}$. Let now $k\geq 2$ and let $\omega\in\R$ with $(\omega,\infty)\subset\Omega$. Then the resolvent identity \eqref{pseudoresid} yields for all $\lambda\in (\omega,\infty)$
\[
R(\lambda)R(\omega)^{k-1}=\frac{R(\omega)-R(\lambda)}{\lambda-\omega}R(\omega)^{k-2}=\tfrac{1}{\lambda-\omega}R(\omega)^{k-1}-\tfrac{1}{\lambda-\omega}R(\lambda)R(\omega)^{k-2},
\]
and a repeated application leads to 
\begin{align}
\label{repeated}
R(\lambda)R(\omega)^{k-1}&=\sum_{i=1}^{k-1}(-1)^{i+1}(\lambda-\omega)^{-i}R(\omega)^{k-i}+(-1)^{k-1}(\lambda-\omega)^{-k+1} R(\lambda).
\end{align}
Now $\mathbf{(D_k)}$ implies
\begin{align*}
&~~~~\|(\lambda-\omega)^{-k+2} R(\lambda)\|\\&\leq \|(\lambda-\omega)R(\lambda)R(\omega)^{k-1}\|+\sum_{i=1}^{k}(-1)^{i+1}(\lambda-\omega)^{-i}\left\|R(\omega)^{k+1-i}\right\|\\&\leq M \|R(\omega)^{k-1}\|+\sum_{i=1}^{k}(-1)^{i+1}(\lambda-\omega)^{-i}\left\|R(\omega)^{k+1-i}\right\|,
\end{align*}
which shows that $\mathbf{(G_k)}$ holds for $k\geq 2$, i.e., (a) is shown.\\
Now we prove (b). The relations $\overline{\ran R(\mu)^{k+1}}=\overline{\ran R(\mu)^k}$,
    $\ker R(\mu)^{k+1}=\ker R(\mu)^k$ follow by a~combination of (a) with Proposition~\ref{prop:ker_stat}. The remaining statement \eqref{empty_int} follows, if 
show that 
\begin{equation}\ker R(\omega)^\ell\cap\overline{\ran R(\omega)^k}=\{0\}\label{eq:RellRk}\end{equation}
by induction on $\ell\in\N$.
To show that \eqref{eq:RellRk} holds for $\ell=1$, we  Let $(\lambda_n)$ be the sequence from $\mathbf{(G_k)}$. We use \eqref{repeated} to conclude that
\begin{align*}
\lim\limits_{n\rightarrow\infty}\|(I-\lambda_nR(\lambda_n))R(\omega)^k\|=0.
\end{align*}
This together with $\mathbf{(D_k)}$ gives  
\begin{align*}
\lim\limits_{n\rightarrow\infty}\lambda_n R(\lambda_n)y=y\quad \forall~y\in\overline{\ran R(\omega)^k}
\end{align*}
and therefore we conclude that \eqref{eq:RellRk} holds for $\ell=1$.
For the induction step, assume that \eqref{eq:RellRk} holds for $\ell\geq1$. Let
$y\in\ker R(\omega)^{\ell+1}\cap\overline{\ran R(\omega)^k}$.
Now, by $R(\omega)y\in\ker R(\omega)^{\ell}$ together with \[R(\omega)y\in R(\omega)\, \overline{\ran R(\omega)^k}\subseteq \overline{\ran R(\omega)^{k+1}}=\overline{\ran R(\omega)^k},\]
we obtain
\[R(\omega)y\in \ker R(\omega)^{\ell+1}\cap\overline{\ran R(\omega)^k}=\{0\}.\]
Consequently, $y\in\ker R(\omega)\cap\overline{\ran R(\omega)^k}=\{0\}$, and the induction statement is shown. 

We continue with the proof of (c). Since the restricted pseudo-resolvent $R|_{\overline{\ran R(\mu)^k}}:\Omega\rightarrow L(\overline{\ran R(\mu)^k})$ fulfills the pseudo-resolvent identity, it is again a pseudo-resolvent and hence by Proposition~\ref{prop:relationfrompsres} we have that the subspace on the right hand side in \eqref{ahat} is closed and does not depend on the choice of $\mu$. Furhtermore, (b) implies that $R(\mu)x=0$ holds only for $x=0$ and therefore the set on the right hand side in \eqref{ahat} defines a closed operator which is denoted by $A_R$.  \end{proof}
The proof of \eqref{empty_int} is based on the general ideas in \cite{Kato59}, where the case $k=1$ has been treated. For this case, a similar construction to obtain an operator $A_R$ was given in \cite[Sec.~5]{ThalThal99}. The operator $A_R$ will be used to construct a continuous semigroup on $\overline{\ran R(\lambda)^{k}}$ which leads to a degenerate semigroup on $X$. The existence of such a degenerate semigroup for $k=1$ was shown previously in~\cite[Thm.~2.3]{FaviYagi99}. Furthermore, the generation of semigroups on the space of consistent initial values was recently investigated in \cite{Tros19}.

In Theorem \ref{thm:main} we do not obtain the desired decoupling \eqref{finite_dec} of the underlying Banach space $X$ for $k>1$. For infinite dimensional spaces such a decomposition was shown in \cite{SvirFedo03} and \cite{TrosWaur19} under additional assumptions.

In \cite{SvirFedo03} 
it was assumed that $X$ is reflexive and that the pseudo-resolvents $R_r$ and $R_l$ given by \eqref{def:RrRl} on $\Omega$ are \emph{weakly $(E,p)$-radial}, i.e.\ we have $(\omega,\infty)\subseteq\Omega$ for some $\omega\in\R$ and for some $p\geq 0$ and $K>0$ the following holds  for all $\mu_0,\ldots,\mu_p\in(\omega,\infty)$
\begin{align}
\label{weakly_Ep_radial}
\max\left\{\left\|\prod_{j=0}^pR_r(\mu_j)\right\|,\left\|\prod_{j=0}^pR_l(\mu_j)\right\|\right\}\leq K\prod_{j=0}^p(\mu_j-\omega)^{-1}.
\end{align}
This is a~stronger assumption than $\mathbf{(D_{p+1})}$, where we consider only $\mu_0=\lambda$ and $\mu_1=\ldots=\mu_p=\omega$ are fixed. 

In \cite{TrosWaur18,TrosWaur19} it was assumed that $X$ is a Hilbert space and that $E$ and $A$ are bounded and fulfill certain closedness assumptions. If $0\in\rho(A)$ then it was required in \cite{TrosWaur18,TrosWaur19}  that $\ran (A^{-1}E)^k$ is closed for some $k\geq 1$ to obtain  a decomposition similar to \eqref{finite_dec}. If want to apply the results in \cite{TrosWaur18,TrosWaur19} to operator pencils $sI-A$ with unbounded $A$, then we can rewrite the pencil as  $sA^{-1}-I$ which has bounded coefficients $\hat E=A^{-1}$ and $\hat A=I$. However assuming e.g.\ the closedness of $\ran (\hat A^{-1}\hat E)=\ran A^{-1}=\dom A$ implies that $A$ is bounded.

In the following we consider the special case when $X$ is a Hilbert space and derive a staircase form of a given  pseudo-resolvent and hence of the associated linear relation given by Proposition~\ref{prop:relationfrompsres}. Note that this approach works without the additional assumptions $\mathbf{(G_k)}$ and $\mathbf{(D_k)}$.  

For $k\geq 1$ and a pseudo-resolvent $R:\Omega\rightarrow L(X)$ where $X$ is now a Hilbert space and some $\mu\in\Omega$ we consider the orthogonal decomposition
\begin{align}
\label{ortho_sum}
X=\mathcal{V}_1\oplus\mathcal{W}_1:=\overline{\ran R(\mu)}\oplus\ker (R(\mu))^*.
\end{align}
Since $R(\mu)(\overline{\ran R(\mu)})\subseteq \overline{\ran R(\mu)}$, the pseudo resolvent can - by using the canonical identification $\mathcal{V}_1\oplus\mathcal{W}_1\cong \mathcal{V}_1\times\mathcal{W}_1$ - be decomposed with respect to \eqref{ortho_sum} as 
\begin{align*}
R(\lambda)=\begin{bmatrix}
R(\lambda)|_{\overline{\ran R(\mu)}}&R(\lambda)|_{\ker R(\mu)^*}\\
0& P_{\ker R(\mu)^*}R(\lambda)|_{\ker R(\mu)^*}\end{bmatrix}=\begin{bmatrix}
R(\lambda)|_{\mathcal{V}_1}&R(\lambda)|_{\mathcal{W}_1}\\
0& 0
\end{bmatrix}
\end{align*}
where we used \eqref{ortho_sum} in the last step.

We repeat the above construction with $R(\lambda)|_{\mathcal{V}_1}$ and consider the orthogonal projection in the Hilbert space  $\mathcal{V}_1=\overline{\ran R(\mu)}$ onto $\mathcal{V}_2:=\overline{\ran R(\mu)^2}$. Then we can decompose 
\begin{align*}
\mathcal{V}_1=\overline{\ran R(\mu)}&=P_{\mathcal{V}_2}\overline{\ran R(\mu)}\oplus(I-P_{\mathcal{V}_2})\overline{\ran R(\mu)}=:\mathcal{V}_2\oplus\mathcal{W}_2,
\end{align*}
which leads to 
\begin{align*}
R(\lambda)|_{\mathcal{V}_1}=\begin{bmatrix}
R(\lambda)|_{\mathcal{V}_2}&R(\lambda)|_{\mathcal{W}_2}\\
0& P_{\mathcal{W}_2}R(\lambda)|_{\mathcal{W}_2}\end{bmatrix}=\begin{bmatrix}
R(\lambda)|_{\mathcal{V}_2}&R(\lambda)|_{\mathcal{W}_2}\\
0& 0
\end{bmatrix}.
\end{align*}
If $\overline{\ran R(\mu)^2}\neq\overline{\ran R(\mu)^3}$ then this process can be continued for $k\geq 3$. The result is an upper block triangular operator matrix decomposition of $R(\lambda)$ in the space
\[
X=\mathcal{V}_k\oplus\mathcal{W}_k\oplus\ldots\oplus\mathcal{W}_{1},\quad \mathcal{V}_k:=\overline{\ran R(\mu)^k}
\]
which is - by canonically identifying $\mathcal{V}_k\oplus\mathcal{W}_k\oplus\ldots\oplus\mathcal{W}_{1}\cong \mathcal{V}_k\times\mathcal{W}_k\times\ldots\times\mathcal{W}_{1}$ of the form
\begin{align}
\label{staircase}
R(\lambda)=\begin{bmatrix}
  R(\lambda)|_{\mathcal{V}_k} & R(\lambda)|_{\mathcal{W}_k} & P_{\mathcal{V}_k}R(\lambda)|_{\mathcal{W}_{k-1}} &\ldots &P_{\mathcal{V}_k}R(\lambda)|_{\mathcal{W}_1} \\ 0&0& P_{\mathcal{W}_k}R(\lambda)|_{\mathcal{W}_{k-1}} &\ldots& P_{\mathcal{W}_k}R(\lambda)|_{\mathcal{W}_1}\\ &\ddots~~~~~&&\ddots&\\ & &~~~~~\ddots&& P_{\mathcal{W}_2}R(\lambda)|_{\mathcal{W}_1}\\& & & 0~~~~&0
\end{bmatrix}
\end{align}

If $\mathbf{(D_k)}$ or $\mathbf{(G_k)}$ hold then according to Theorem~\ref{thm:main} we have that  $R(\lambda)|_{\mathcal{V}_k}$ is injective and hence it coincides with the resolvent of the operator $A_R$ constructed in Theorem~\ref{thm:main}~(c).

Without further assumptions on the operators this might not lead to an injective pseudo-resolvent as the following example shows. 
\begin{example}
Consider the left shift operator $E:\ell^2(\N,\C)\rightarrow\ell^2(\N,\C)$ which is given by  $(x_1,x_2,\ldots,)\mapsto E(x_1,x_2,\ldots):=(x_2,x_3,\ldots)$ and $A=I_{\ell^2}$. This leads to the pseudo-resolvent $R(\lambda):=E(I-\lambda E)^{-1}$ which exists in a neighborhood of $\lambda=0$. Furthermore, $\ran R(0)^k=\ell^2(\N,\C)$ holds for all $k\geq0$ which implies that $R(0)|_{\ran R(0)^k}=E$ is not injective for all $k\geq 0$. 
\end{example}

For Hilbert spaces we present another sufficient condition for the existence of an operator $A_R$ as in Theorem \ref{thm:main}~(c) which is easier to verify than $\mathbf{(G_k)}$ and $\mathbf{(D_k)}$. In the following we assume that $0\in\rho(A)$ holds. If
    \begin{align}
\label{y_impli}
    y\in\ker E\cap\dom A\quad  \land \quad Ay\in\ran E \Longrightarrow y=0,
    \end{align}
then
\begin{align}
\label{dierel}
{\rm gr \,} A_R=\begin{smallbmatrix}EA^{-1}\\I_Z\end{smallbmatrix}(\overline{\ran E}) \subseteq \overline{\ran E}\times \overline{\ran E}
\end{align}
is the graph of a closed operator $A_R$. This coincides with \eqref{ahat} for $\mu=0$. To show \eqref{y_impli} directly, let $EA^{-1}z=0$ for some $z\in \overline{\ran E}$, then $A^{-1}z\in\dom A\cap\ker E$ and $AA^{-1}z=z\in\overline{\ran E}$. Hence by assumption $z=0$ which implies that~\eqref{dierel} is not multi-valued. Since $\overline{\ran E}$ is closed and $EA^{-1}$ is bounded, \eqref{dierel} is also closed and hence the operator is closed.

\section{Index concepts for abstract DAEs}
\label{sec:indexsol}
In this section, we discuss different index concepts for infinite dimensional differential-algebraic equations
\begin{align}
\label{adae_subsec}
\tfrac{\rm d}{{\rm d}t}Ex(t)=Ax(t),\quad  t\geq 0,
\end{align}
where $E:X\rightarrow Z$ is bounded, $X$, $Z$ are Banach spaces, and $A:X\supset \dom A\rightarrow Z$ is densely defined and closed. In the following, we will associate with~\eqref{adae_subsec} the pseudo-resolvent given by the resolvent of the linear relation
\[
\begin{smallbmatrix}
E\\A
\end{smallbmatrix}\,\cdot\,\dom A=AE^{-1}\subseteq Z\times Z,\quad E^{-1}:=\{(x,Ex)~|~x\in X\}^{-1}\subseteq Z\times X,
\]
where $E$ is identified with its graph and the inverse is taken in the sense of linear relations. Hence a pseudo-resolvent is given by
\begin{align}
\label{rangepres}
\Omega:=\rho(AE^{-1}),\quad R_{l}(\lambda)=(AE^{-1}-\lambda)^{-1}=E(A-\lambda E)^{-1}.
\end{align}
Based on this, the ADAE~\eqref{adae_subsec} is said to have \emph{pseudo-resolvent index} $k\in\N$, if  $(\omega,\infty)\subseteq\rho(E,A)$ for some $\omega\in\R$ and $k$ is the smallest natural number such that $R_{l}$ fulfills $\mathbf{(G_k)}$.

An alternative definition can be obtained using the linear relation $E^{-1}A$ which has the resolvent
\[
R_{r}(\lambda)=(E^{-1}A-\lambda)^{-1}=(A-\lambda E)^{-1}E,\quad \lambda\in\rho(E,A).
\]
We first show that non-emptiness of the resolvent together with boundedness of $E$ and completeness of $X$ and $Z$ implies that $A$ is closed. We further show that the graph norm
\begin{equation}\|x\|_{\dom A_0}:=(\|x\|_X^2+\|A_0x\|_X^2)^{1/2},\label{eq:graphnorm} \end{equation}
is equivalent to two further norms associated to $E$ and $A$.

\begin{proposition}\label{prop:normeq}
Let $X$, $Z$ be Banach spaces, let $E:X\to Z$, $A:X\supset\dom A\to Z$ be linear, and assume that $\rho(E,A)\neq\emptyset$ and $(\lambda E-A)^{-1}\in L(Z,X)$ for some $\lambda \in\rho(E,A)$. Then $A$ is closed, and $(\lambda E-A)^{-1}\in L(Z,X)$ for all $\lambda \in\rho(E,A)$. Further, for all $\lambda\in\rho(E,A)$ there exist $c_0,c_1,c_2>0$, such that
\begin{multline}
    \forall\,x\in\dom A:\quad \|x\|_{\dom A}\leq c_0\|(\lambda E-A)x\|_{Z}\\\leq c_1\left(\|Ex\|_Z+\|Ax\|_Z\right)\leq c_2  \|x\|_{\dom A}.\label{eq:eqnorms}
\end{multline}
\end{proposition}
\begin{proof}
{\em Step 1:} We show that $A$ is closed. Boundedness of $(\lambda E-A)^{-1}$ yields closedness of $\lambda E-A$, and boundedness of $E$ now leads to closedness of $A$.\\
{\em Step 2:} We show that $(\lambda E-A)^{-1}\in L(Z,X)$ for all $\lambda \in\rho(E,A)$. This follows by a~combination of the already proven closedness of $A$ with the closed graph theorem \cite[Thm.~7.9]{Alt16}.\\
{\em Step 3:} We show that there exists some $c_0>0$, such that, for all $x\in\dom A$,
\[\|x\|_{\dom A}\leq c_0\|(\lambda E-A)x\|_{Z}.\]
Assuming the converse, we obtain that there exists some sequence $(x_n)$ in $\dom A$ with
$\|(\lambda E-A)x_n\|_Z\to 0$ and $\|x_n\|_X+\|Ax_n\|_X\to 1$. Boundedness of $(\lambda E-A)^{-1}$ yields that
\[\|x_n\|_X=\|(\lambda E-A)^{-1}(\lambda E-A)x_n\|_X\to 0,\]
whence $\|Ax_n\|_Z\to 1$, and, by boundedness of $E$, $\|(\lambda E-A)x_n\|_Z\to 1$, which is a~contradiction.\\
{\em Step 4:} We show that there exists some $\tilde{c}_1>0$, such that
\[\|(\lambda E-A)x\|_{Z}\leq \tilde{c}_1\left(\|Ex\|_Z+\|Ax\|_{Z}\right).\]
This is a~simple consequence of the triangle inequality and can be achieved by the choice
\[\tilde{c}_1=\max\left\{1,\tfrac1\lambda\right\}.\]
{\em Step 5:} We show that there exists some $\tilde{c}_2>0$, such that, for all $x\in\dom A$,
\[\|Ex\|_Z+\|Ax\|_{Z}\leq \tilde{c}_2 \|x\|_{\dom A}.\]
This holds, since we have
\[\|Ex\|_Z+\|Ax\|_{Z}\leq \|E\|\,\|x\|_X+\|Ax\|_Z\leq  \tilde{c}_2 \|x\|_{\dom A}\]
for
\[\tilde{c}_2:=\sqrt2\cdot\max\{1,\|E\|\}.\]
{\em Step 6:} We conclude that \eqref{eq:eqnorms} holds:  The first inequality in \eqref{eq:eqnorms} has been proven in step~3, whereas the remaining two ones follow from the statements shown in steps 3 and 4 by setting
\[c_1=\tilde{c}_1 \,c_0,\quad c_2=\tilde{c}_2 \,c_2.\]
\end{proof}

For more details on the spectrum of regular operator pencils and the associated linear relations $E^{-1}A$ and  $AE^{-1}$ we refer to \cite{GernMoal20}. 
In the remainder of this section, we compare the pseudo-resolvent index~\eqref{rangepres} with other index assumptions in the literature. An overview is given in Figure~\ref{fig:index}.

 \begin{figure}
 {\footnotesize
 \begin{center}
   \resizebox{0.6\linewidth}{!}{\begin{tikzpicture}[thick,node distance = 22ex,implies/.style={double,double equal sign distance,-implies}, box/.style={fill=white,rectangle, draw=black},
   blackdot/.style={inner sep = 0, minimum size=3pt,shape=circle,fill,draw=black},plus/.style={fill=white,circle,inner sep = 0,very thick,draw},
   metabox/.style={inner sep = 2ex,rectangle,draw,dotted,fill=gray!20!white}]

     \node (diss)  [metabox,minimum size=1.2cm]  {\parbox{6cm}{\hspace{0cm} $\begin{matrix} \text{$R_{l}(\lambda)=E(A-\lambda E)^{-1}$ fulfills $\mathbf{(D_k)}$ }\\ \text{and  $(\omega,\infty)\subseteq\rho(E,A)$} \end{matrix}$
}};

     \node (psri)     [below of = diss,metabox,minimum size=1.2cm,yshift=5ex]  {\parbox{6cm}{\hspace{0cm}
     \text{~~ADAE has finite pseudo-resolvent index}}};

     \node (ri)     [below of = psri, metabox,minimum size=1.2cm,yshift=5ex]  {\parbox{6cm}{\hspace{0cm}\text{~~~~~ADAE has finite resolvent index}}};

         \node (ti)     [below of = ri, metabox,minimum size=1.2cm,yshift=5ex]  {\parbox{6cm}{\hspace{0.3cm}$\begin{matrix}\text{ADAE has finite tractability index and}\\ \exists M,\omega >0\,\forall\lambda>\omega: ~\|(\lambda I-\mathfrak{U})^{-1}\|\leq M \end{matrix}$}};

 \node (psri2) [left of = psri, xshift = 10ex, yshift = -4ex] {};
 \node (ri2) [left of = ri, xshift = 10ex, yshift = 4ex] {};

 \node (psri3) [right of = psri, xshift = -10ex, yshift = -4ex] {};
 \node (ri3) [right of = ri, xshift = -10ex, yshift = 4ex] {};

 \node (ri4) [left of = ri, xshift = 10ex, yshift = -4ex] {};
 \node (ti2) [left of = ti, xshift = 10ex, yshift = 4ex] {};

 \node (ri5) [right of = ri, xshift = -10ex, yshift = -4ex] {};
 \node (ti3) [right of = ti, xshift = -10ex, yshift = 4ex] {};


 \draw[->,semithick,double,double equal sign distance,>=stealth] (diss)   -- node [right] {}  (psri);


    \draw[<-,semithick,double,double equal sign distance,>=stealth] (psri3)  -- node [right] {Proposition~\ref{prop:res_index}~~}  (ri3);


     \draw[->,semithick,double,double equal sign distance,>=stealth] (ti)  -- node [right] {Proposition~\ref{prop:tithenri}~~}  (ri);


 \node (ri6) [left of = ri, xshift = 14ex, yshift = -4ex] {};
 \node (ti4) [left of = ti, xshift = 14ex, yshift = 4ex] {};



   \draw[<-,semithick,double,double equal sign distance,>=stealth] (ri2)  -- node [left] {}  (psri2);

 \end{tikzpicture}}
 \end{center}}
   \caption{An overview of the different index concepts.}
   \label{fig:index}
  \end{figure}
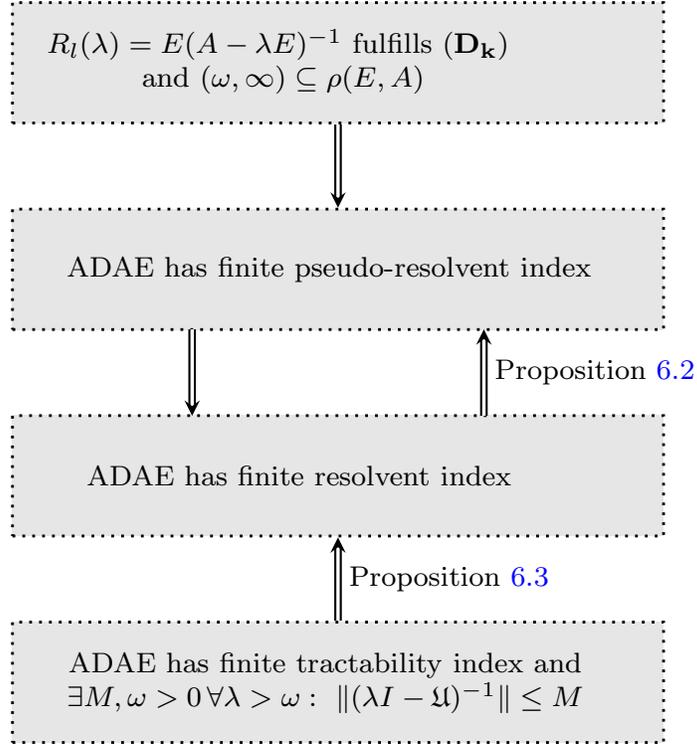

\subsection{Comparison with resolvent index}
In~\cite{TrosWaur18,TrosWaur19} it was assumed that the DAE satisfies for some $k\in\N$
\begin{align}
\label{TWass}
\Cpo\subseteq \rho(E,A), \quad \|(\lambda E-A)^{-1}\|\leq M|\lambda|^{k-1}\,\,\text{for all $\lambda\in\Cpo$}.
\end{align}
The smallest such $k$ will be called the \emph{resolvent index} of the ADAE. Note that in \cite{TrosWaur18} the condition \eqref{TWass} was stated with $k$ instead of $k-1$. However with this definition the index would be different from the commonly used index for linear DAEs on finite dimensional spaces \cite{KunkMehr06,LamoMarz13}.

Below, we consider the weaker conditions
\begin{align*}
&\mathbf{(R_k)} \quad \exists M\geq 1, ~\exists \omega\in\R:\quad \|\lambda^{-k+1}(\lambda E-A)^{-1}\|\leq M,\quad \forall\lambda\in(\omega,\infty)\\
&\mathbf{(R_k^w)} \quad \exists M\geq 1\text{ and a~sequence } (\lambda_n)_n\text{ in $\Omega\cap\R$ with } \lambda_n\to \infty:\\&\qquad\qquad  \|\lambda_n^{-k+1}(\lambda_n E-A)^{-1}\|\leq M.
\end{align*}
Clearly, $\mathbf{(R_k)}$ for some $k\geq 0$ implies $\mathbf{(R_k^w)}$. From Theorem~\ref{thm:main} we know that $\mathbf{(D_k)}$ implies $\mathbf{(G_k)}$. In the proposition below, we compare the growth conditions  $\mathbf{(R_k)}$, $\mathbf{(R_k^w)}$, and $\mathbf{(G_k)}$. 
\begin{proposition}
\label{prop:res_index}
Let $X,Z$ be a Banach spaces and $E:X\rightarrow Z$ bounded, $A:X\supseteq\dom A\rightarrow Z$ closed and densely defined and assume that $(\omega,\infty)\subseteq\rho(E,A)$, $k\geq 1$. Then the following holds.
\begin{itemize}
\item[\rm (a)] If $R_{l}$ fulfills $\mathbf{(G_k)}$ then $\mathbf{(R_k^w)}$ holds.
\item[\rm (b)] If $\mathbf{(R_k^w)}$ holds then $R_{l}$ fulfills $\mathbf{(G_{k+1})}$. Further, $R_{l}$ does not fulfill $\mathbf{(G_{k-1})}$.
\item[\rm (c)] If $A$ is bounded and $\mathbf{(R_k^w)}$ holds then $R_{l}$ fulfills $\mathbf{(G_k)}$.

\item[\rm (d)] If $A$ is bounded and $\mathbf{(R_k)}$ holds then $R_{l}$ fulfill $\mathbf{(D_k)}$. 
\end{itemize}
\end{proposition}
\begin{proof}
An essential ingredient in the proof is the resolvent identity
\begin{align}
\label{residpencil2}
(\mu-\lambda)(A-\mu E)^{-1}E(A-\lambda E)^{-1}&=(A-\mu E)^{-1}-(A-\lambda E)^{-1}
\end{align}
which holds for all $\mu,\lambda\in\rho(E,A)$. 
\begin{itemize}
\item[\rm (a)]
Assume that $R_{l}$ satisfies $\mathbf{(G_k)}$ for some $k\geq 1$. Then $\mathbf{(R_k^w)}$ follows, since, for some $\mu\in\rho(E,A)$, \eqref{residpencil2} gives
\begin{align*}
M&\geq \|\lambda_n^{-k+2}E(A-\lambda_n E)^{-1}\|\\&=\|\lambda_n^{-k+2}(A-\mu E)(A-\mu E)^{-1}E(A-\lambda_n E)^{-1}\|\\
&\geq \|(A-\mu E)^{-1}\|^{-1}\|\lambda_n^{-k+2}(\mu-\lambda_n)^{-1}((A-\mu E)^{-1}-(A-\lambda_n E)^{-1})\|.
\end{align*}
\item[\rm (b)]
Assume that $\mathbf{(R_k^w)}$ holds. Then, by
\[
\|E(A-\lambda E)^{-1}\|\leq \|E\|\|(A-\lambda E)^{-1}\|,
\]
we obtain that $\mathbf{(G_{k+1})}$ holds. Furthermore,
 \begin{align*}
  \|E(A-\lambda E)^{-1}\|&\geq \|(A-\mu E)^{-1}\|^{-1}\|(A-\mu E)^{-1}E(A-\lambda E)^{-1}\|\\&=\|(A-\mu E)^{-1}\|^{-1}\|(\mu-\lambda)^{-1}((A-\mu E)^{-1}-(A-\lambda E)^{-1})\|,
 \end{align*}
which means that $R_l$ does not fulfill $\mathbf{(G_{k-1})}$. 
\item[\rm (c)] Assume that $A$ is bounded. Then \eqref{residpencil2} leads to
 \begin{align}
\label{A_bdd_estim}
 \begin{split}
 \|E(A-\lambda E)^{-1}\|&= \|(A-\mu E)(A-\mu E)^{-1}E(A-\lambda E)^{-1}\|\\&\leq \|A-\mu E\|\|(\mu-\lambda)^{-1}((A-\mu E)^{-1}-(A-\lambda E)^{-1})\|.
 \end{split}
 \end{align}
 This implies that $R_{l}$ fulfills $\mathbf{(G_k)}$.
\item[\rm (d)] Assume that $\mathbf{(R_k)}$ holds for $k\geq 1$ and that $A$ is bounded.
 If $k=1$ then \eqref{A_bdd_estim} together with $\mathbf{(R_1)}$ implies $\mathbf{(D_1)}$.  If $k\geq 2$ and let $x\in \ran R_{l}(\mu)^{k-1}$. Then $x=E(A-\mu E)^{-1}y$ for some $y\in\ran R_{l}(\mu)^{k-2}$, and we obtain
 \begin{align*}
 \|E(A-\lambda E)^{-1}x\|&=
 \|E(A-\lambda E)^{-1}E(A-\mu E)^{-1}y\|\\&=\|E(\lambda-\mu)^{-1}((A-\mu E)^{-1}-(A-\lambda E)^{-1})y\|.
 \end{align*}

 Repeating the argument for $k\geq 3$, $y=E(A-\mu E)^{-1}z$ and $z\in \ran R_{l}(\mu)^{k-3}$, another application of \eqref{A_bdd_estim} shows that $\mathbf{(D_k)}$ holds. Further,  $\mathbf{(R_k)}$ implies
 \[
 \|\lambda^{-k+1}(\lambda E'-A')^{-1}\|=\|\lambda^{-k+1}(\lambda E-A)^{-1}\|\leq M,
 \]
 i.e.\ $(E',A')$ fulfills $\mathbf{(R_k)}$.
 As~a~consequence, \[R_{l}'(\lambda)=(E(A-\lambda E)^{-1})'=(A'-\lambda E')^{-1}E'=((E')^{-1}A'-\lambda)^{-1}\] 
 fulfills $\mathbf{(D_k)}$. \end{itemize}
 \end{proof}
Note that it is straightforward to show that an analogous statement to Proposition~\ref{prop:res_index} holds for right resolvent $R_r(\lambda)=(A-\lambda E)^{-1}E$.

Note that for unbounded $A$, it is in principle possible to replace $X$ with $\dom A$ equipped with the norm $\|x\|_{\mu E-A}=\|(\mu E-A)x\|_Z$ for some $\mu\in\rho(E,A)$. Then, for all $\lambda\in\C$,
\begin{align*}
\|(\lambda E-A)^{-1}\|_{\mu E-A}:=\|(\mu E-A)(\lambda E-A)^{-1}\| \phantom{:}=\|I+(\mu-\lambda)E(\lambda E-A)^{-1}\|. \end{align*}
This can be used to show equality between the index of $R_{AE^{-1}}$ and the resolvent index. Namely, by 
\[
|\lambda^{-k+1}|\|I+(\mu-\lambda)E(\lambda E-A)^{-1}\|=\|\lambda^{-k+1}(\lambda E-A)^{-1}\|_{\mu E-A}\leq M
\]
we are led to the existence of some $\hat M$ with
\[
\|\lambda^{-k+2}E(\lambda E-A)^{-1}\|\leq \hat M.
\]
In other words, $\mathbf{(G_k)}$ holds. We briefly remark that, by Proposition~\ref{prop:normeq}, the norm $\|\cdot\|_{\mu E-A}$ is equivalent to the graph norm of $A$.

\subsection{Comparison with tractability index}
As another index concept for infinite dimensional DAEs, we recall the \emph{tractability} \emph{index} from \cite{ReisTisc05,Reis06}, see also \cite{Marz96}. For Hilbert spaces $X$ and $Z$, assume that $E\in L(X,Z)$ has closed range, and $A:\dom A\to Z$, such that $\rho(E,A)\neq\emptyset$. The ADAE \eqref{adae} is said to have \textit{tractability index} $k$, if the following sequence exists  
\begin{equation}\label{eq:chain}\begin{aligned}
&E_0:=E,& &A_0:=A,\\
&Q_i\in L(X)\cap L(\dom A),& &Q_i^2=Q_i, P_i^2=P_i,\\
&\ran Q_i=N_i:=\ker E_i, &&\ker Q_i=\ran P_i\supset\sum_{j=0}^{i-1}N_j\\
&E_{i+1}:=E_i-A_iQ_i,&&\dom E_{i+1}=\dom E_i\cap(\dom A_i+\ker Q_i)\\
&A_{i+1}:=A_iP_i, &&\dom A_{i+1}= \dom A_i+\ran Q_i,
\end{aligned}
\end{equation}
and $N_{k}=\{0\}$, and, if $k\geq1$, $N_{k-1}\neq\{0\}$. It is shown in~\cite[Thm.~4.1]{ReisTisc05} and  that ADAEs with tractability index $k$, and the additional property that 
$\ran E+A\sum_{i=0}^{k-1}N_i$ is closed, can be transformed into an upper triangular form 
\begin{align}
\label{decoupling}
WET=\begin{bmatrix}N&0\\0&I\\0&0\end{bmatrix}\in L(X_1\times X_2,X_1\times X_2\times X_3),\quad WAT=\begin{bmatrix}I&K\\0&\mathfrak{U}\\0&\Gamma\end{bmatrix}
\end{align}
for some Hilbert spaces $X_1,X_2,X_3$ and $W\in L(Z,X_1\times X_2)$ injective with dense range, $T\in L(X_1\times X_2,X)$ bijective, some closed and densely defined operators $\mathfrak{U}:X_2\supseteq \dom\mathfrak{U}\rightarrow X_2$, $K:X_2\supseteq \dom K\rightarrow X_1$, $\Gamma:X_2\supseteq \dom\Gamma\rightarrow X_3$, and some nilpotent $N\in L(X_1)$ with index $k$, i.e.\ $N^k=0$ and $N^{k-1}\neq 0$. It has been shown in \cite[Sec.~2.2.1]{Reis06} that there exist (even practically motivated) operator pencils which are not equivalent to a block diagonal pencil. In other words, it is not always possible to achieve the above form \eqref{decoupling} in which $K=0$. Further note that, if additionally
\begin{equation}\label{eq:kerEi}
    \ker E_i\subset \dom A,\quad i=0,\ldots, k-1,
\end{equation}
then $E_k:X\to Z$ is bounded, and the construction of $W$ in \cite{ReisTisc05} yields that $W$ is bijective, that is, also the inverse of $W$ is bounded. It follows further from 
\cite[Thm.~2.8]{Reis06}
that 
\begin{equation}\label{eq:KL}
    K\in L\Big(\dom\left[\begin{smallmatrix}\mathfrak{U}\\\Gamma\end{smallmatrix}\right],X_1\Big).
\end{equation}
The transformed pencil \eqref{decoupling} corresponds to a~triangular ADAE of the form
\begin{align}
\label{adae_reg}
\ddts \begin{bmatrix}
    N&0\\0&I\\0&0
\end{bmatrix}\begin{pmatrix}x_1(t)\\x_2(t)\end{pmatrix}=\begin{bmatrix}
     I&K\\0&\mathfrak{U}\\0&\Gamma
\end{bmatrix}\begin{pmatrix}x_1(t)\\x_2(t)\end{pmatrix}+\begin{pmatrix} f_1(t)\\
    f_2(t)\\f_3(t)
\end{pmatrix},\quad  \begin{pmatrix}
    Nx_1(0)\\ x_2(0)
\end{pmatrix}=\begin{pmatrix}
    Nx_{10}\\ x_{20}
\end{pmatrix},
\end{align}
for 
\[\begin{pmatrix}x_1(t)\\x_2(t)\end{pmatrix}=T^{-1}x(t),\; \begin{pmatrix}f_1(t)\\f_2(t)\\f_3(t)\end{pmatrix}=Wf(t),\; \begin{pmatrix}x_{10}\\x_{20}\end{pmatrix}=T^{-1}x_0.\]
Since $N$ is nilpotent with index $k$, we obtain that for all $\lambda\in\C$,
\begin{align}
\label{Ninv}
(\lambda N-I)^{-1}=-I-\ldots-(\lambda N)^{k-1}.
\end{align}
Hence, by further invoking \eqref{eq:KL}, regularity of the ADAE
 \eqref{adae} implies regularity of the ``lower right block'' of \eqref{adae_reg} with, in particular,
\[
\rho(E,A)=\rho\left(\begin{bmatrix}
    I\\ 0
\end{bmatrix},\begin{bmatrix}
   \mathfrak{U} \\ \Gamma
\end{bmatrix}\right)
\]
and
\begin{equation}
\label{eq:decres2}
(\lambda WET-WAT)^{-1}=\begin{pmatrix}(\lambda N-I)^{-1}~~&(\lambda N-I)^{-1}K\begin{bmatrix}\lambda I-\mathfrak{U}\\ -\Gamma\end{bmatrix}^{-1}\\0&\begin{bmatrix}\lambda I-\mathfrak{U}\\ -\Gamma\end{bmatrix}^{-1}\end{pmatrix},
\end{equation}
Note that the resolvent of 
$\Big(\Big[\begin{smallmatrix}
    I\\ 0    
\end{smallmatrix}\Big],\Big[\begin{smallmatrix}
    \mathfrak{U}\\ \Gamma    
\end{smallmatrix}\Big]
\Big)$
has the form
\begin{equation}\label{eq:lrres}
    \begin{bmatrix}
   sI-\mathfrak{U} \\ -\Gamma
\end{bmatrix}^{-1}=\begin{bmatrix}
   (sI-\mathfrak{U}|_{\ker\Gamma})^{-1} \,&\, -\Gamma|_{\ker sI-\mathfrak{U}}^{-1}
\end{bmatrix}\in L(X_2\times X_3,X_2).
\end{equation}
Further, by \eqref{eq:KL}, we have 
\[K\begin{bmatrix}\lambda I-\mathfrak{U}\\ -\Gamma\end{bmatrix}^{-1}\in L(X_2\times X_3,X_1)\quad\forall\,\lambda\in\rho(E,A).\]

In the following proposition we show that - under a certain additional assumption - a finite tractability index implies that $\mathbf{(R_{k_t})}$ and $\mathbf{(G_{k_t})}$ hold. 

\begin{proposition}
\label{prop:tithenri}
Consider an ADAE~\eqref{adae_subsec} with tractability index $k\in\N$. Assume that the operators in~\eqref{decoupling} satisfy
\begin{align*}
\exists~\omega>0:\quad (\omega,\infty)\subseteq\rho\left(\begin{bmatrix}
    I\\ 0
\end{bmatrix},\begin{bmatrix}
   \mathfrak{U} \\ \Gamma
\end{bmatrix}\right),\text{ and }\quad 
\sup_{\lambda>\omega}\left\|\begin{bmatrix}
   \lambda-\mathfrak{U} \\ \Gamma
\end{bmatrix}^{-1}\right\|< \infty,\quad \sup_{\lambda>\omega}\left\|K\begin{bmatrix}
   \lambda-\mathfrak{U} \\ \Gamma
\end{bmatrix}^{-1}\right\|< \infty.
\end{align*}
then $(\omega,\infty)\subseteq\rho(E,A)$ and $k$ is the smallest number such that $\mathbf{(R_{k})}$ holds. If, additionally, \eqref{eq:kerEi}, $k$ is the smallest number such that $\mathbf{(G_{k})}$ holds.
\end{proposition}
\begin{proof}
Let $\lambda\in(\omega,\infty)$. The bounded invertibility of $\lambda E-A$ follows from \eqref{eq:decres2}. 
Define the constants
\[M_1=\|W\|\,\|T\|,\quad M_2:= \sup_{\lambda>\omega}\left\|\begin{bmatrix}
   \lambda I-\mathfrak{U} \\ \Gamma
\end{bmatrix}^{-1}\right\|,\quad M_3:= \sup_{\lambda>\omega}\left\|K\begin{bmatrix}
   \lambda I-\mathfrak{U} \\ \Gamma
\end{bmatrix}^{-1}\right\|\]
Then, by using \eqref{eq:decres2}, we obtain for all $\lambda\in (\omega,\infty)$ that
\begin{align*}
\|(\lambda E-A)^{-1}\|&=\|T(\lambda WET-WAT)^{-1}W\|\\
&\leq M_1\|(\lambda WET-WAT)^{-1}\|\\&\leq M_1(\|(\lambda N-I)^{-1}\|+M_2+M_3).
\end{align*}
Then \eqref{Ninv} immediately implies $\mathbf{(R_{k})}$. 
Next, we show that for $k$ is the smallest number such that $\mathbf{(R_{k})}$ holds. If $k=0$, then nothing has to be shown. Hence, we assume that $k\ge1$, and we show that  $\mathbf{(R_{k-1})}$ does not hold. Since the projector chain \eqref{eq:chain} does not stagnate until $k$, we have $\ker E_{k-1}\neq0$. Let $z\in \ker E_{k-1}$. Then the construction of $W$ in \cite{ReisTisc05} yields that $z\in\ran W$ with 
\[W^{-1}z=\begin{pmatrix}z_1\\0\\0\end{pmatrix}\]
for some $z_1\in \ker N^{k-1}$. This gives, by invoking \eqref{eq:decres2}, 
\begin{align*}
\|(\lambda E-A)^{-1}z\|&\geq \|T^{-1}\|^{-1}\|(\lambda WET-WAT)^{-1}W^{-1}z\|\\&=\|T^{-1}\|^{-1}\|(\lambda N-I)^{-1}z_1\|\\
&=\|T^{-1}\|^{-1}\left\|\sum_{i=0}^{k-1}\lambda^iN^{i}z_1\right\|\\
&=\|T^{-1}\|^{-1}\lambda^{k-1}\left\|\left(\sum_{i=0}^{k-1}\lambda^{i-k+1}N^{i}z_1\right)+N^{k-1}z_1\right\|.
\end{align*}
Since the latter factor does not converge to zero for $\lambda\to\infty$, we have that $\|(\lambda E-A)^{-1}z\|$ cannot be bounded on $(\omega,\infty)$ by a polynomial of degree smaller than $k-1$. This means that $\mathbf{(R_{k-1})}$ is not fulfilled.\\
Now assume that further \eqref{eq:kerEi} holds. To show $\mathbf{(G_{k})}$ is true, and that $k$ is the smallest number with this property, we first make use of the fact that, 
by \eqref{eq:decres2} and \eqref{eq:lrres},
\begin{align*}
    E(\lambda E-A)^{-1}&=W^{-1}(WET)(\lambda WET-WAT)^{-1}W\\
&=W^{-1}\begin{bmatrix}N(\lambda N-I)^{-1}&(\lambda N-I)^{-1}K(\lambda I-\mathfrak{U}|_{\ker\gamma})^{-1}\\0& (\lambda I-\mathfrak{U}|_{\ker\gamma})^{-1}\end{bmatrix}W.
\end{align*}
Since, as previously stated, \eqref{eq:kerEi} implies boundedness of $W^{-1}$, and, further, $(\lambda N-I)^{-1}=-\sum_{i=1}^{k-1}\lambda^{i-1}N^{i}$, an application of an argumentation completely analogous to the one in the first part of the proof yields the desired result.
 \end{proof}

\section{Dissipativity-based conditions}
\label{sec:diss}
In this section, we present some sufficient conditions for a~closed linear relation $L$ fulfilling $\mathbf{(D_k)}$. These conditions are based on the {\em $\omega$-dissipativity condition}
\begin{align}
\label{dissip}
\forall\lambda>\omega:\quad \|y-\lambda x\|\geq (\lambda-\omega)\|x\|,\quad  \forall\, (x,y)\in L.
\end{align}
Note that the latter implies that $\ker(L-\lambda)=\{0\}$ for all $\lambda>\omega$.\\
First note that, by the Hahn-Banach theorem, for $x\in X$, there exists some $x'\in X'$ with $\|x'\|^2=\|x\|^2=\scpr{x'}{x}_{X',X}$. In other words, the {\em duality set}
\[\Jc(x)=\setdef{x'\in X'}{\|x'\|^2=\|x\|^2=\scpr{x'}{x}_{X',X}}\]
is non-empty. In the case where $X$ is a~Hilbert space, we have $\Jc(x)=\{x\}$, where we use the identification $X\cong X'$ by the Riesz representation theorem.

Below, we provide a characterization of the $\omega$-dissipativity condition and a sufficient condition and for a half-axis to be in the resolvent set. 
\begin{proposition}
\label{prop:oktober}
Let $X$ be a Banach space, $\omega\in\R$, and let $L$ be a closed linear relation in $X$. 
\begin{enumerate}[\rm(a)]
\item\label{prop:oktobera} $L$ fulfills the $\omega$-dissipativity condition \eqref{dissip} if, and only if,
\begin{align}
\label{dissip_HS}
\forall\, (x,y)\in L\;\;\exists\,x'\in\Jc(x)\,:\quad \re\scpr{x'}{y}_{X',X}\leq\omega\|x\|^2.
\end{align}
\item\label{prop:oktoberb} Assume that $L$ fulfills the $\omega$-dissipativity condition \eqref{dissip}, and there exists some $\lambda_0>\omega$, such that $L-\lambda_0$ is surjective. Then $(\omega,\infty)\subseteq\rho(L)$.
\item\label{prop:oktoberc} Assume that $(\omega,\infty)\subseteq\rho(L)$. Then $L$ fulfills the $\omega$-dissipativity condition \eqref{dissip} if, and only if, its pseudo-resolvent $R:\rho(L)\to L(X)$ of $L$ fulfills
    \begin{align*}
\forall\,\lambda>\omega:\quad   \|R(\lambda)\|\leq \frac{1}{\lambda-\omega}.
    \end{align*}
\end{enumerate}
\end{proposition}
\begin{proof}
By considering $L-\omega$ instead of $L$, it is no loss of generality to assume that $\omega=0$ in the sequel.
\begin{itemize}
\item[\rm (a)] First assume that~\eqref{dissip_HS} holds for $\omega=0$. Let $(x,y)\in L$. Since, in the case where $x=0$, \eqref{dissip} trivially holds, we can assume that $x\neq 0$. Let $\lambda>0$, and let $y'_\lambda\in\Jc(y-\lambda x)$, and set $z_\lambda'=y_\lambda'/\|y_\lambda'\|$. Then 
\begin{align*}
\lambda \|x\|
&\leq \|y-\lambda x\|\\
&= \scpr{y-\lambda x}{z_\lambda'}\,\|x\|\\
&= \re\scpr{y}{z_\lambda'}-\lambda\re\scpr{x}{z_\lambda'}\\
&\leq\min\{\lambda \|x\|-\re\scpr{y}{z_\lambda'},\lambda\re\scpr{x}{z_\lambda'}+\|y\|\}.
\end{align*}
This gives 
\[\re\scpr{y}{z_\lambda'}\leq0,\quad \|x\|-\tfrac1\lambda \|y\|\leq \re\scpr{x}{z_\lambda'}.\]
By Alaoglu's theorem \cite[Sec.~8.7]{Alt16}, there exists some $z'\in X'$ which is an accumulation point of the sequence $(z_n')$. Then $\|z'\|\leq1$, and 
\[\re\scpr{y}{z'}\leq0,\quad \|x\|\leq \re\scpr{x}{z'}.\]
Hence, $x':=\|x\|\cdot z'\in\Jc(x)$, and
\[\re\scpr{y}{x'}\leq0,\]
which shows that \eqref{dissip_HS} is true for $\omega=0$.\\
For the reverse implication, assume that \eqref{dissip_HS} holds for $\omega=0$. Since \eqref{dissip_HS} is trivially fulfilled for $x=0$, we can now assume that $x$ is nonzero. Then
\begin{align*}
\|y-\lambda x\| \|x\|\geq  \scpr{\lambda x-y}{x'}
=\lambda \|x\|^2-\scpr{y}{x'}
\geq \lambda \|x\|^2.
\end{align*}
This shows that \eqref{dissip} holds for $\omega=0$.

\item[(b)] Assume that $L$ is $0$-dissipative $L-\lambda_0$ is surjective for some $\lambda_0>0$. Since $0$-dissipativity directly implies that $\ker L-\lambda_0=\{0\}$, we obtain that $(L-\lambda_0)^{-1}$ can be regarded as a~bounded operator $R(\lambda_0)\in L(X)$, and \eqref{dissip} (with $\omega=0$) moreover yields that $\|R(\lambda_0)\|\leq\tfrac1\lambda_0$. Hence, the power series  \eqref{eq:powser}
converges for all $\lambda\in(0,2\lambda_0)$. By the Neumann series \cite[Sec.~5.7]{Alt16}, we have 
\[R(\lambda)=R(\lambda_0)(I-(\lambda_0-\lambda)R(\lambda_0))^{-1}.\]
Now let $x\in X$. Then
for $y=(I-(\lambda_0-\lambda)R(\lambda_0))^{-1}x$, we have
\begin{align*}
(R(\lambda)x,x)&=(R(\lambda_0)(I-(\lambda_0-\lambda)R(\lambda_0))^{-1}x,x)\\
&=(R(\lambda_0)y,(I-(\lambda_0-\lambda)R(\lambda_0))y)\\
&=(R(\lambda_0)y,y-(\lambda_0-\lambda)R(\lambda_0))y)\in R(\lambda_0)^{-1}+(\lambda_0-\lambda)\\&=(L-\lambda_0)+(\lambda_0-\lambda)\\&=L-\lambda.
\end{align*}
Hence, $L-\lambda$ is surjective for all $\lambda\in(0,2\lambda_0)$.
Repeating this argumentation for $3/2\lambda_0$, we obtain that 
$L-\lambda$ is surjective for all $\lambda\in(0,3\lambda_0)$. Proceeding in this way, we see that $L-\lambda$ is surjective for all $\lambda>0$.
\item[(c)]  By rephrasing $x\rightsquigarrow R(\lambda)y$, $y-\lambda x\rightsquigarrow x$ in \eqref{dissip}, we obtain that $\omega$-dissipativity is equivalent to
    \begin{align*}
\forall\,\lambda>\omega:\quad   \|y\|\geq (\lambda-\omega)\|R(\lambda)y\|,\quad  \forall\, y\in X,
    \end{align*}
and the desired result follows immediately.
\end{itemize}
\end{proof}
If the linear relation is represented by an operator pencil $\lambda E-A$ via  $L=L_l=AE^{-1}$, then \eqref{dissip} is equivalent to
\begin{equation}
\label{eq:EAdiss}
\forall\lambda>\omega:\quad \|(\lambda E-A)x\|\geq (\lambda-\omega)\|Ex\|,\quad  \forall\, x\in\dom A,
\end{equation}
which is, by Proposition~\ref{prop:oktober}\,\eqref{prop:oktobera}, equivalent to 
\[
\forall\, x\in X\;\;\exists\,z'\in\Jc(Ex)\,:\quad \re\scpr{z'}{(A-\omega E)x}_{X',X}\leq0.
\]
If, additionally, $X$ is a Hilbert space, the latter is equivalent to 
\[
\re(Ex,(A-\omega E)x)_X\leq 0,\quad \forall x\in \dom A,
\]
where, hereby, $(\cdot,\cdot)_X$ stands for the inner product in $X$. Note that ADAEs with this property have recently been studied in \cite{JacoMorr22}.

As a consequence of the above findings, we can formulate a~result on regularity of operator pencils belonging to a~dissipative relation, and give an estimate for the left resolvent. The latter has also been obtained for the case of Hilbert spaces \cite[Thm.~3.6]{JacoMorr22}. 

\begin{corollary}
\label{cor:two}
Let $X$, $Z$ be Banach spaces, let $E\in L(X,Z)$, and let $A:X\supseteq\dom A\rightarrow Z$ be closed and densely defined, such that $L_l=AE^{-1}$ is $\omega$-dissipative for some $\omega\in\R$.\\
If $\ker E\cap\ker A=\{0\}$, and $\lambda E-A$ is surjective for some $\lambda>\omega$, then $(\omega,\infty)\subset\rho(E,A)$, and 
\begin{align}
\label{RAE_estim}
\forall\lambda>\omega:\quad \|E(A-\lambda E)^{-1}\|\leq\frac{1}{\lambda-\omega}.
\end{align}
In particular, the left resolvent $\lambda\mapsto E(A-\lambda E)^{-1}$ fulfills $\mathbf{(D_1)}$ with $M=1$.
\end{corollary}
\begin{proof}
 Proposition~\ref{prop:oktober}\,\eqref{prop:oktoberb} implies that $(\omega,\infty)\subseteq\rho(L)$ and hence 
 $\lambda E-A$ is surjective for all $\lambda>\omega$. Furthermore, \eqref{eq:EAdiss} together with $\ker E\cap\ker A=\{0\}$ yields $\ker\lambda E-A=\{0\}$ for all $\lambda>\omega$. Thus,  $(\omega,\infty)\subseteq\rho(E,A)$. Further, \eqref{RAE_estim} follows from  Proposition~\ref{prop:oktober}~\eqref{prop:oktoberc}. 
 \end{proof}
As a~consequence of the above result, \eqref{dissip_HS}
implies the existence of an $\omega$-stable degenerate semigroup, see also \cite{JacoMorr22} for the case of Hilbert spaces.
\begin{corollary}
Under the assumptions of Corollary~\ref{cor:two} and, further, $Z$ is locally sequentially weakly 
compact, it holds that
\begin{equation}\label{eq:Zdecomp}
Z=\overline{\ran E(A-\mu E)^{-1}}\oplus\ker E(A-\mu E)^{-1}\quad\forall\, \mu>\omega.
\end{equation}
Moreover, the operator $A_R$ as in \eqref{ahat} generates an $\omega$-stable semigroup $(T(t))_{t\geq 0}$ and
\[
T_R(t)(z_R+z_K)=T(t)z_R,\quad z_R\in\overline{E(A-\mu E)^{-1}Z},~ z_K\in \ker E(A-\mu E)^{-1}
\]
with, in particular,
\begin{equation}
\|T_R(t)\|\leq e^{-\omega t}\quad\forall\, t>0.\label{eq:omegacontr}    
\end{equation}
\end{corollary}
\begin{proof}
Since, by Corollary~\ref{cor:two}, $\mathbf{(D_k)}$ is fulfilled, we can, as stated in \eqref{k=1_dec}, apply the results from \cite{Kato59} to see that the decomposition \eqref{eq:Zdecomp} holds.\\
Moreover, the estimate \eqref{RAE_estim} implies
\[
\|(A_R-\lambda I)^{-n}\|\leq\|((AE^{-1}-\lambda)^{-1})^n\|\leq\|(AE^{-1}-\lambda)^{-1}\|^n\leq\tfrac{1}{(\lambda-\omega)^n}
\]
and hence, by the Feller-Miyadera-Phillips theorem \cite[Thm.~II.3.8]{EngeNage06} $A_R$ generates a~strongly continuous semigroup $T_R(\cdot)$ with \eqref{eq:omegacontr}.
 \end{proof}
In the previous considerations we have analyzed operator pencils that correspond to dissipative linear relations. There might be several situation in which this is a~too strong assumption though the corresponding ADAE ``behaves dissipatively'' is some sense. In the following, for a~linear relation with non-empty resolvent, we relax the dissipativity assumption to dissipativity on a~subspace formed by the closure of a~power of the resolvent.

\begin{proposition}
    Let $X$ be a Banach space, $\omega\in\R$, and let $L$ be a closed linear relation in $X$, and assume that there exists some $\omega\in\R$ with $[\omega,\infty)\subset\rho(L)$.
   Then the following statements are equivalent for the pseudo-resolvent $R:\rho(L)\to L(X)$ of $L$:
    \begin{itemize}
\item[\rm (i)] 
\[\forall\, (x,y)\in L\text{ with }x\in \ran R(\omega)^{k-1} \;\;\exists\,x'\in\Jc(x)\,:\quad \re\scpr{x'}{y}_{X',X}\leq\omega\|x\|^2.\]
\item[\rm (ii)] 
\[\forall \lambda >\omega, \forall x\in \ran R(\omega)^{k-1}:\quad 
\|R(\lambda)x\|\leq \frac{ \|x\|}{\lambda-\omega}.\]
\end{itemize}\end{proposition}
\begin{proof}
This follows from Proposition~\ref{prop:oktober}\,\eqref{prop:oktobera} by replacing $L$ with 
\[L_k:= L\cap \big(\overline{\ran R(\omega)^{k-1}}\times \overline{\ran R(\omega)^{k-1}}\big).\]
\end{proof}

We use the latter to consider the case where $X$ is a~Hilbert space, and the linear relation is formed by the left resolvent of an operator pencil with self-adjoint and nonnegative $E\in L(X)$,
that is
\[\forall\,x\in X:\quad (x,Ex)_X\geq 0,\]
$\ran E$ is closed, 
and $A$ such that $A-\omega E$ is a~dissipative operator. We show in the last section that this covers several practically important ADAEs.

\begin{proposition}
\label{prop:suff_for_D2}
Assume that $X$ is a Hilbert space, and $E\in L(X)$ is self-adjoint, nonnegative and has closed range. Further, let 
$A:X\supset \dom A\rightarrow X$ be linear, and assume that there exists some $\omega\in\R$ such that $A-\omega E$ is dissipative, i.e.\ \[\re\big((A-\omega E)x,x\big)_X\leq 0 \quad\forall x\in\dom A.\] Moreover, assume that $\ker E\cap\ker A=\{0\}$ and $\lambda_0 E-A$ is surjective for some $\lambda_0>\omega$. Then $\lambda E-A$ has a~bounded inverse for all $\lambda>\omega$, and the left resolvent $R_{l}$ fulfills $\mathbf{(D_2)}$ with $\omega$ as above, and
\begin{equation}\label{eq:M12def}
\begin{aligned}
M&=\frac{M_1}{M_2},\text{ where}\\
M_1&=\sup\setdef{\|Ex\|}{x\in\ran E,\, \|x\|_X=1},\\
M_2&=\inf\setdef{\|Ex\|}{x\in\ran E,\, \|x\|_X=1},\\
\end{aligned}\end{equation}
\end{proposition}
\begin{proof}
First we consider the norm
\[\|x_1+x_2\|_E^2:=(x_1,E^+x_1)_X+\|x_2\|_X^2\quad \forall x_1\in\ran E,\,x_2\in\ker E,\]
where $E^+$ is the Moore-Penrose inverse of $E$.
Since $E$ is self-adjoint, nonnegative and has closed range, this norm is equivalent to the standard norm on $X$ and, for the constants in \eqref{eq:M12def},
\begin{equation}
    \forall\,x\in\ran E:\qquad M_1^{-1}\|x\|_X\leq \|x\|_E\leq M_2^{-1}\|x\|_X.\label{eq:normestimate}
\end{equation}
{\em Step 1:} We show that $\lambda E-A$ is injective for all $\lambda>\omega$:
Assume that $\lambda>\omega$, $x\in X$, such that $(\lambda E-A)x=0$. Then
\[0=\re(x,(\lambda E-A)x)_X=(\lambda-\omega) (x,Ex)_X+\re(x,(A-\omega E)x)_X.\]
Since both summands are nonnegative, we obtain that $(x,Ex)_X=0$, and self-adjointness together with nonnegativeness of $E$ leads to $Ex=0$, whence also $Ax=0$. Then, by the assumption on the nullspaces of $E$ and $A$, we have $x=0$.

{\em Step 2:} We show that $\lambda_0 E-A$ has a~bounded inverse. Since, for all $x\in \dom A$,
\[\re(x,(\lambda_0 E-A)x)_X=(\lambda_0-\omega) (x,Ex)_X+\re(x,(A-\omega E)x)_X\geq0,\]
we have that $A-\lambda_0 E$ is dissipative. As $A-\lambda_0 E$ is moreover assumed to be surjective, it is closed by \cite[Prop.~3.14\,(iii)]{EngeNage06}. Combining the latter with the already proven injectivity of $\lambda_0 E-A$, we obtain that $\lambda_0 E-A$ has a~bounded inverse.

{\em Step 3:} We show that 
$\begin{smallbmatrix} 
E(\lambda_0 E-A)^{-1}\\A(\lambda_0 E-A)^{-1}
\end{smallbmatrix}\cdot\ran E$ is a $\omega$-dissipative relation on $\ran E$ provided with the norm $\|\cdot\|_E$. First we state that, by
\begin{equation}A(\lambda_0 E-A)^{-1}=-I_X+\lambda_0E(\lambda_0 E-A)^{-1},\label{eq:resreldiss}\end{equation}
$\begin{smallbmatrix} 
E(\lambda_0 E-A)^{-1}\\A(\lambda_0 E-A)^{-1}
\end{smallbmatrix}\cdot\ran E$ is a~relation on $\ran E$. 
By another use of \eqref{eq:resreldiss}, we obtain that for all $x\in \ran E$,
\begin{align*}
\re(E(\lambda_0 E-A)^{-1}x,A(\lambda_0 E-A)^{-1}x)_E=\re((\lambda_0 E-A)^{-1}x,A(\lambda_0 E-A)^{-1}x)_X\leq0.
\end{align*}
which shows that the relation is $\omega$-dissipative.

{\em Step 4:} We show that $\lambda E-A$ has a~bounded inverse for all $\lambda>\omega$.  
Injectivity of $\lambda E-A$ has already been proven in the first step. 
To we show that $\lambda E-A$ is surjective, we first
combine step~3 with \eqref{dissip} to obtain that 
\[\forall\,x\in\ran E:\quad\|E(\lambda_0 E-A)^{-1}x\|_E\leq\tfrac{\|x\|_E}{\lambda_0-\omega}.\]
Hence, for all $\lambda\in(\omega,\omega+2(\lambda_0-\omega))$, the power series 
\eqref{eq:powser}
converges in the operator norm induced by $\|\cdot\|_E$, and thus, by \eqref{eq:normestimate}, it also converges in the conventional operator norm. By using the result \cite[Sec.~5.7]{Alt16} on Neumann series, we have 
\begin{align*}
(\lambda E-A)R(\lambda)&=(\lambda E-A)(\lambda_0 E-A)^{-1}(I-(\lambda_0-\lambda)E(\lambda_0 E-A)^{-1})^{-1}\\
&=(\lambda E-A)(\lambda_0 E-A-(\lambda_0-\lambda)E)\\ &=I_Z.    
\end{align*}
This implies that $\lambda E-A$ is surjective for all $\lambda\in(\omega,\omega+2(\lambda_0-\omega))$.
Now an argumentation as in step~2 yields that 
$\lambda E-A$ has a~bounded inverse for all $\lambda\in(\omega,\omega+2(\lambda_0-\omega))$
Repeating our argumentation for $\omega+\tfrac32(\lambda_0-\omega)$ instead of $\lambda_0$, we obtain that 
$\lambda E-A$ has a~bounded inverse for all $\lambda\in(\omega,\omega+3(\lambda_0-\omega))$. Proceeding in this way, we see that $\lambda E-A$ is has a~bounded inverse for all $\lambda>\omega$.

{\em Step 5:} We show that
the left resolvent $R_{l}$ fulfills $\mathbf{(D_2)}$ with $\omega$ and $M$ as stated in the assumptions.
We have seen in step~3 that
$\begin{smallbmatrix} 
E(\lambda_0 E-A)^{-1}\\A(\lambda_0 E-A)^{-1}
\end{smallbmatrix}\cdot\ran E$ is a $\omega$-dissipative relation on $\ran E$ provided with the norm $\|\cdot\|_E$.
Then  \eqref{dissip} yields that for all $x\in \ran E$,
\begin{align*}
\|E(A-\lambda E)^{-1}x\|_E
=\|E(A-\lambda_0 E)^{-1}\,((A-\lambda_0 E)^{-1}A-\lambda (A-\lambda_0 E)^{-1}E)x\|_E
\leq \frac{\|x\|_E}{\lambda-\omega}.    
\end{align*}
Now using \eqref{eq:normestimate}, we obtain that
\[\forall\,x\in\ran E:\quad \|E(A-\lambda E)^{-1}x\|_X\leq \frac{M\,\|x\|_X}{\lambda-\omega},\]
and the proof is complete.
 \end{proof}

The statement below demonstrates that a specific category of differential-algebraic boundary control systems fulfills condition $\mathbf{(D_k)}$. For further insights into classical boundary control systems, see e.g.~\cite[Chap.~11]{JacoZwar12}. 

\begin{proposition}\label{prop:bnddiss}
Let $X,Y$ be Hilbert spaces, and assume that $E_0\in L(X)$ is self-adjoint, nonnegative and has closed range. Further, let 
$A_0:X\supset \dom A_0\rightarrow X$, $\Gamma\in L(\dom A_0,Y)$ be surjective with $\dom A_0$ equipped with the graph norm \eqref{eq:graphnorm}, and we assume that there exists 
some $\omega\in\R$ such that $A_0|_{\ker \Gamma}-\omega E_0$ is dissipative, and some
$\lambda_0>\omega$, such that $\lambda_0 E_0-A_0|_{\ker \Gamma}$ is surjective.
Moreover, assume that $\ker E_0\cap\ker A_0\cap\ker\Gamma=\{0\}$.\\
Then, for $E\in L(X,Z)$, $A:X\supset \dom A\rightarrow Z$ with
\[
E=\begin{smallbmatrix}
E_0\\0
\end{smallbmatrix},\quad A=\begin{smallbmatrix}
A_0\\ \Gamma
\end{smallbmatrix},\quad Z=X\times Y,
\]
it holds that $(\omega,\infty)\subseteq \rho(E,A)$ and the left resolvent $R_{l}$ fulfills $\mathbf{(D_2)}$. If, moreover, $\ker E_0=\{0\}$, then the right resolvent $R_r$ fulfills $\mathbf{(D_1)}$.
\end{proposition}

\begin{proof}
{\em Step 1:} We show that $(\omega,\infty)\subseteq \rho(E,A)$ with, for all $\lambda>\omega$,
\begin{equation} (\lambda E-A)^{-1}=\begin{bmatrix}
    (\lambda E_0-A_0|_{\ker \Gamma})^{-1}&-\Gamma^-(\lambda)
\end{bmatrix}\label{eq:bndinv}\end{equation}
for some $\Gamma^-(\lambda)\in L(Y,X)$.\\
By Proposition~\ref{prop:suff_for_D2}, we have
$(\omega,\infty)\subseteq \rho(E_0,A_0|_{\ker\Gamma})$. In particular, $\lambda_0 E_0-A_0$ is closed. Thus, by boundedness of $E_0$, $A_0$ is closed as well, i.e., $\dom A_0$ is a~Hilbert space.
Now surjectivity of $\Gamma\in L(\dom A_0,Y)$ together with closed graph theorem \cite[Thm.~7.9]{Alt16} implies that its Moore-Penrose inverse $\Gamma^+:Y\to\dom A_0$ is bounded with $\Gamma\Gamma^+=I_Y$, and we obtain that 
\[\forall\,\lambda>\omega:\quad \Gamma^-(\lambda):=\Gamma^+-(\lambda E_0-A_0|_{\ker \Gamma})^{-1}(\lambda E_0-A_0)\Gamma^+ \in L(Y,X).\]
Then we obtain
\[\Gamma\Gamma^-(\lambda)=\Gamma\Gamma^+-\underbrace{\Gamma(\lambda E_0-A_0|_{\ker \Gamma})^{-1}}_{=0}(\lambda E_0-A_0)\Gamma^+=I_Y\]
 and, by invoking
\[(\lambda E_0-A_0)(\lambda E_0-A_0|_{\ker \Gamma})^{-1}=I_X\]
and
\begin{align*}(\lambda E_0-A_0)\Gamma^-(\lambda)=(\lambda E_0-A_0)\Gamma^+-\underbrace{(\lambda E_0-A_0)(\lambda E_0-A_0|_{\ker \Gamma})^{-1}}_{=I_X}(\lambda E_0-A_0)\Gamma^+=0\end{align*}
we see that, for all $\lambda>\omega$,
\[\begin{bmatrix}
    \lambda E_0-A_0\\-\Gamma
\end{bmatrix}\begin{bmatrix}
    (\lambda E_0-A_0|_{\ker \Gamma})^{-1}&-\Gamma^-(\lambda)
\end{bmatrix}=I_{Z}.\]
This shows that $(\omega,\infty)\subseteq \rho(E,A)$ and that \eqref{eq:bndinv} holds.\\
{\em Step 2:} We show that the left resolvent $R_{l}$ fulfills $\mathbf{(D_2)}$. By using the representation \eqref{eq:bndinv} of the resolvent, we obtain that the left resolvent fulfills
\[\forall\,\lambda>\omega:\quad R_l(\lambda)=E(\lambda E-A)^{-1}=\begin{bmatrix}
E_0(\lambda E_0-A_0|_{\ker \Gamma})^{-1}&E_0\Gamma^-(\lambda)\\0&0
\end{bmatrix}.\]
Now using that $E=\left[\begin{smallmatrix}E_0\\0\end{smallmatrix}\right]$ and, by Proposition~\ref{prop:suff_for_D2}, $E_0(\lambda E_0-A_0|_{\ker \Gamma})^{-1}$ fulfills $\mathbf{(D_2)}$, we obtain that $R_l$  fulfills $\mathbf{(D_2)}$ as well.\\
 {\em Step 3:} We show, under the additional assumption that 
$\ker E_0=\{0\}$ holds, the right resolvent fulfills $\mathbf{(D_1)}$.\\
Since $E_0\in L(X)$ is assumed to be self-adjoint, nonnegative, and it has closed range, it follows from $\ker E_0=\{0\}$ that $E_0$ has a~bounded inverse. By using step~1, we see that the right resolvent fulfills, for all $\lambda>\omega$,
\begin{equation}(\lambda E-A)^{-1}E=(\lambda E_0-A_0|_{\ker \Gamma})^{-1}E_0
=E_0^{-1} \big(E_0(\lambda E_0-A_0|_{\ker \Gamma})^{-1}\big)E_0
.\label{eq:ressimil}\end{equation}
It follows from Proposition~\ref{prop:suff_for_D2} that 
\[\lambda\mapsto E_0(\lambda E_0-A_0|_{\ker \Gamma})^{-1}\]
fulfills $\mathbf{(D_2)}$. By invoking that, by bijectivity of $E_0$,
\[\forall\,\lambda>\omega:\quad \overline{E_0(\lambda E_0-A_0|_{\ker \Gamma})^{-1}}=X,\]
we obtain that it even fulfills $\mathbf{(D_1)}$. Now invoking \eqref{eq:ressimil}, it follows that the right resolvent
$R_r(\lambda)=(\lambda E-A)^{-1}E$ fulfills $\mathbf{(D_1)}$. 
 \end{proof}

\begin{remark}\
\begin{enumerate}[(a)]
\item In Proposition~\ref{prop:bnddiss}, we have made the assumption that $\Gamma\in L(\dom A_0,Y)$, where $\dom A_0$ is equipped with the graph norm. Proposition~\ref{prop:normeq} implies that this is equivalent to the existence of a positive constant $c_1>0$ such that
\[
\forall\, x\in\dom A:\quad \|\Gamma x\|_Y\leq c_1\,\big(\|Ex\|_Z+\|Ax\|_Z\big).
\]
By using Proposition~\ref{prop:normeq} once more, we can see that this is equivalent to the property that for some (and therefore any) $\lambda>\omega$, there exists a positive constant $c_2>0$ such that
\[
\forall\, x\in\dom A:\quad \|\Gamma x\|_Y\leq c_2\,\|(\lambda E-A)x\|_Z,
\]
which is in turn equivalent to the condition that $\Gamma(\lambda E-A)^{-1}\in L(Z,Y)$ for some (and therefore any) $\lambda\in\rho(E,A)$.

\item Surjective mappings $\Gamma$ of the  type as in Proposition~\ref{prop:bnddiss} are typically defined through boundary triplets or quasi-boundary triplets, especially in the context of partial differential equations. For a comprehensive overview, you can refer to~\cite{BehrHassdeSn20}.
\item Matrix pencils with symmetric positive semidefinite $E=E^*\in\mathbb{C}^{n\times n}$ and dissipative $A\in\mathbb{C}^{n\times n}$ are subject of \cite{BergReis14b}, where it has been exploited that the resolvent is positive real, that is, $(A-\lambda E)^{-1}$ is dissipative for all $\lambda$ in the right complex half-plane. \\
This type of matrix pencils is further referred to as ``semi-dissipative Hamiltonian'' in \cite{Achl22}. DAEs corresponding to such pencils exhibit strong connections to port-Hamiltonian systems, as discussed in \cite{GernHall20,MehlMehrWojt18}.
\end{enumerate}\end{remark}

\section{Solutions of abstract differential-algebraic equations}
\label{sec:pseudoDGL}
In this section, we consider for $f\in L^1([0,t_f];Z)$ and some $t_f>0$ regular ADAEs
\begin{align}
\label{inhom}
\tfrac{\rm d}{{\rm d}t}Ex(t)=Ax(t)+f(t)\quad\text{for a.a.\ $t\in[0,t_f]$,}\quad (Ex)(0)=Ex_0,
\end{align}
where $E:X\rightarrow Z$ is bounded between reflexive Banach spaces $X$, $Z$, and $A:X\supset \dom A\rightarrow Z$ is densely defined and closed. Throughout this section, we will exploit the relation of solutions of the ADAE~\eqref{inhom} with the pseudo-resolvent differential equation that is for a pseudo-resolvent $R:\Omega\rightarrow L(X)$ and $\mu\in\Omega$ given by 
\begin{align}
    \label{psresDGL_allg}
    \tfrac{\rm d}{{\rm d} t} R(\mu)w_{\mu}(t)=w_{\mu}(t)+\hat f_{\mu}(t),\quad R(\mu)w_{\mu}(0)=R(\mu)w_0.
\end{align}

The following relation between the sets of solutions of \eqref{inhom} and \eqref{psresDGL_allg} follows from Lemma~\ref{lem:shift}.

\begin{lemma}
\label{lem:shift_ext}
Consider the ADAE \eqref{inhom} and the pseudo-resolvent equation \eqref{psresDGL_allg} and let $\mu\in\rho(E,A)$.  Then the following holds:
\begin{itemize}
    \item[\rm (a)] 
The function $x$ is a classical (weak, mild) solution of \eqref{inhom} if and only if $w_{\mu}=(A-\mu E)e^{-\mu\cdot}x(\cdot)$ is a classical (weak, mild) solution of \eqref{psresDGL_allg} with $R=R_l$ and $w_0=(A-\mu E)x_0$ and $\hat f_{\mu}(t)=e^{-\mu t}f(t)$.

    \item[\rm (b)] 
The function $x$ is a classical (weak, mild) solution of \eqref{inhom} if and only if $w_{\mu}=e^{-\mu\cdot}x(\cdot)$ is a classical (weak, mild) solution of \eqref{psresDGL_allg} with $R=R_r$ and $w_0=x_0$ and $\hat f_{\mu}(t)=(A-\mu E)^{-1}e^{-\mu t}f(t)$.

\end{itemize}
\end{lemma}
As a consequence of Lemma~\ref{lem:shift_ext} it is sufficient we will study in the following instead of the solutions of \eqref{inhom} only those of \eqref{psresDGL_allg}.

In the proposition below, we show the existence and uniqueness of classical and mild (weak) solutions for the homogeneous equation, i.e.\ \eqref{inhom} with $f=0$.
\begin{proposition}
\label{prop:homo}
Let $R:\Omega\rightarrow L(X)$ be a pseudo-resolvent and consider \eqref{psresDGL_allg} with $\hat f_{\mu}=0$ and $\mu\in\Omega$. Assume that $R$ fulfills $\mathbf{(D_k)}$ and let $A_{R}$ be the operator constructed in Theorem~\ref{thm:main}. If $A_R$ generates a strongly continuous semigroup $(T(t))_{t\geq 0}$, then for all $w_0\in\overline{\ran R(\mu)^{k}}$ the unique mild (weak) solution of \eqref{psresDGL_allg} is given by
\begin{align}
\label{xmu_mit_T}
w_{\mu}(t)=e^{-\mu t}T(t)w_0,\quad t\in[0,t_f].
\end{align}
Further, if $w_0\in R(\mu)\overline{\ran R(\mu)^{k}}$ then \eqref{xmu_mit_T} is the unique classical solution of \eqref{psresDGL_allg}.
\end{proposition}
\begin{proof}
\emph{Step 1:}
Let $w_0\in R(\mu)\overline{\ran R(\mu)^{k}}$. Then we show that~\eqref{xmu_mit_T} is the unique classical solution of \eqref{psresDGL_allg}. Since $R(\mu)\overline{\ran R(\mu)^{k}}=\dom A_R$, it follows from \cite[Thm.~5.2.2]{JacoZwar12} that $w(t)=T(t)w_0$ is the unique classical solution of $\dot w(t)=A_{R}w(t)$ with $w(0)=w_0$ and $w(t)\in\dom A_{R}$ holds for all $t\in[0,t_f]$.
It remains to show that $w_\mu(t)=e^{-\mu t}w(t)$ is a classical solution of \eqref{psresDGL_allg}. Since $\dot w(t)=A_{R}w(t)$, we have $(w(t),\dot w(t))\in{\rm gr\,} A_{R}$ for all $t\in[0,t_f]$. Applying subspace operations we obtain
\begin{align}
\label{one}
(w(t),\dot w(t))\in {\rm gr \,} A_{R} &\Longleftrightarrow (e^{-\mu t}w(t),e^{-\mu t}\dot w(t))\in {\rm gr\,} A_R\\ & \Longleftrightarrow (\underbrace{e^{-\mu t}w(t)}_{=w_{\mu}(t)},\underbrace{-\mu e^{-\mu t}w(t)+e^{-\mu t}\dot w(t)}_{=\dot w_{\mu}(t)}) \in {\rm gr\,} (A_R-\mu),\nonumber
\end{align}
where we used the differentiability of $w_\mu$ together with the product rule in the last step. 
Moreover, \eqref{ahat} yields
\begin{align}
\label{two}
{\rm gr\,} (A_{R}-\mu)=\{(R(\mu)y,y) ~|~ y\in \overline{\ran R(\mu)^{k}}\}.
\end{align}
Thus, if $w$ is a classical solution then combining~\eqref{one} and~\eqref{two} implies that $w_{\mu}(t)=R(\mu)y(t)$
is the unique classical solution of~\eqref{psresDGL_allg}.\linebreak 
\emph{Step 2:} It is shown that~\eqref{xmu_mit_T} is the unique mild solution of~\eqref{psresDGL_allg} if $w_0\in \overline{\ran R(\mu)^{k}}=\overline{\dom A_{R}}$. Invoking \cite[Thm.~5.2.2]{JacoZwar12}, we have $\int_0^tT(\tau)w_0d\tau\in\dom A_{R}$ and the unique mild solution of $\dot w(t)= A_{R}w(t)$ is given by
\[
w(t)-w(0)=T(t)w_0-w_0= A_{R}\int_0^tT(\tau)w_0d\tau.
\]
Differentiating implies that $(w(t),\dot w(t))\in{\rm gr\,} A_{R}$ holds for almost all $t\in[0,t_f]$.
Thus, integrating~\eqref{one} and using the closedness of $A_{R}-\mu$ yields
\[
\left(\int_0^tw_{\mu}(\tau)d\tau,(w_{\mu}(t)-w_{\mu}(0))\right)\in {\rm gr\,}(A_{R}-\mu),
\]
which can be rewritten with~\eqref{two} as
\[
R(\mu)w_{\mu}(t)-R(\mu)w_{\mu}(0)=\int_0^tw_{\mu}(\tau)d\tau.
\]
Hence $w_{\mu}$ is a mild solution of~\eqref{psresDGL_allg}. The uniqueness of the mild solution follows from the uniqueness of $w$ as a mild solution of $\dot w(t)=A_Rw(t)$, $w(0)=w_0$.
 \end{proof}

Before studying the solution of inhomogeneous ADAEs, we recall several sufficient conditions on $A_R$ to be the generator of a strongly continuous semigroup. Typically this is characterized in terms of the resolvent $(A_R-\lambda)^{-1}=R(\lambda)\vert_{\overline{\ran R(\mu)^{k}}}$ for some $\mu\in\Omega$.  Hence $R$ is a generator if and only if
\begin{itemize}
\item[(i)] $R$ satisfies a~\textit{Feller-Miyadera-Philips condition}, if there exists $\omega\in\R$ with $(\omega,\infty)\subseteq\Omega$ and for some $M\geq 1$
    \begin{align}
\label{FMP}
\forall n\in\N~~\forall \lambda\in(\omega,\infty):\quad  \|(A_R-\lambda)^{-n}\|\leq \frac{M}{(\lambda-\omega)^n}.
\end{align}
\end{itemize}
Another sufficient condition is the boundedness of the resolvent of $A_R$ in a half-plane, see~\cite[Cor.~3.7.12]{ArenBatt11},
\begin{itemize}
\item[(ii)] $A_R$ is \textit{bounded in a~half-plane} if for some $k\in\N\setminus\{0\}$, $\omega\in\R$, it holds $\Cpo\subseteq\Omega$ and for some $M\geq 1$ and
    \begin{align}
\label{BH}
    \|\lambda(A_R-\lambda)^{-1}\|\leq M,\quad \text{for all $\lambda\in\Cpo$.}
    \end{align}
\end{itemize}
Note that $\mathbf{(D_k)}$ with $M=1$ implies \eqref{FMP}. If $\mathbf{(D_k)}$ is replaced by
\[
\|\lambda R(\lambda)x\|\leq M\|x\|,\quad \text{for all $\lambda\in\Cpo$ and all $x\in \ran R(\mu)^{k}$}
\]
for some $M\geq 1$ then \eqref{BH} holds. Furthermore, if \eqref{BH} holds only for all $\lambda\in(\omega,\infty)$ then this is not sufficient for $A_R$ to generate a strongly continuous semigroup, see \cite[Ex.~3.12~(3)]{EngeNage06}.

In the proposition below, we show the existence and uniqueness of classical and mild (weak) solutions for the inhomogeneous equation~\eqref{psresDGL_allg} for sufficiently smooth inhomogeneities. If $X$ is reflexive and the pseudo-resolvent $R$ fulfills $\mathbf{(D_k)}$, then we assume that the inhomogeneity $\hat f_{\mu}$ in \eqref{psresDGL_allg} can be written as
\begin{align}
\label{f_dec}
\hat f_{\mu}=f_R\dot + f_K,\quad  f_R:[0,t_f]\rightarrow \overline{\ran R(\mu)^k},\quad f_K:[0,t_f]\rightarrow \ker R(\mu)^k.
\end{align}

Note that this assumption is fulfilled if the condition \eqref{weakly_Ep_radial} holds, see also \cite[Thm.\ 5.7.1]{SvirFedo03} for an analogous result using this assumption.
\begin{proposition}
\label{prop:inhom}
Let $X$ be a reflexive Banach space and $R:\Omega\rightarrow L(X)$ be a pseudo-resolvent which fulfills $\mathbf{(D_k)}$ for some $k\geq 1$, let $A_{R}$  be the operator constructed in Theorem~\ref{thm:main}, and assume that $A_R$ generates a~strongly continuous semigroup, and assume that $\hat f_{\mu}$ in \eqref{psresDGL_allg} satisfies \eqref{f_dec}. Then the following holds:
\begin{itemize}
\item[\rm (a)] If $f_K\in C^{k}([0,t_f];\ker R(\mu)^k)$ and $f_R\in C([0,t_f];\overline{\ran R(\mu)^k})$ then there is a unique classical solution $x$ with initial value
\begin{align}
\label{IVal_classic}
x(0)=x_0-\sum_{i=0}^{k-1}R(\mu)^if^{(i)}_K(0).
\end{align}
\item[\rm (b)]
If $f_K\in W^{k-1,1}([0,t_f];\ker R(\mu)^k)$ then there is a unique mild solution $x$ with initial value
\begin{align}
\label{IVal}
R(\mu)x(0)=R(\mu)x_0-\sum_{i=0}^{k-1}R(\mu)^{i+1}f^{(i)}_K(0).
\end{align}
\end{itemize}
\end{proposition}
\begin{proof}
\emph{Step 1:} We show that there is a~classical solution $x$ of \eqref{psresDGL_allg}  whose initial value fulfills \eqref{IVal_classic}. The uniqueness of the classical solution follows then from the uniqueness of the homogeneous equation with initial value $x(0)=0$ that was shown in Proposition~\ref{prop:homo}.

We show that \eqref{psresDGL_allg} is solved by $x=x_1+x_2$, with $x_1\in L^1([0,t_f];\overline{\ran R(\mu)^k})$, $x_2\in L^1([0,t_f];\ker R(\mu)^k)$ and
\begin{align}
\label{decoupled}
\tfrac{\rm d}{{\rm d} t}R(\mu)x_1(t)=x_1(t)+f_R(t),\quad  \tfrac{\rm d}{{\rm d} t}R(\mu)x_2(t)=x_2(t)+f_K(t).
\end{align}
Using $x_1(t)=(A_R-\mu)y_1(t)$ for some $y_1\in L^1([0,t_f];\overline{\ran R(\mu)^k})$ in \eqref{decoupled} gives 
\begin{align}
\label{transformed_f_R}
\tfrac{\rm d}{{\rm d} t}y_1(t)=(A_R-\mu)y_1(t)+f_R(t), \quad y_1(0)=y_{1,0}\in\overline{\ran R(\mu)^k}.
\end{align}
Since $A_R$ is the generator of $(T(t))_{t\geq 0}$, $A_R-\mu$ generates the strongly continuous semigroup $(e^{-\mu t}T(t))_{t\geq 0}$. Further, $f_R\in C([0,t_f];\overline{R(\mu)^kX})$, and thus the unique classical solution of~\eqref{transformed_f_R} for  $y_{1,0}\in R(\mu)\overline{\ran R(\mu)^{k-1}}$ is given by, see e.g.~\cite[Thm.~10.1.8]{JacoZwar12},
\begin{align}
\label{VDK}
y_1(t)=e^{-\mu t}T(t)y_{1,0}+\int_0^te^{-\mu(t-\tau)}T(t-\tau)f_R(\tau){\rm d}\tau.
\end{align}

Since $f_K\in C^{k}([0,t_f];\ker R(\mu)^k)$ the following function is well defined 
\begin{align}
\label{x_k_def}
x_2=-\sum_{i=0}^{k-1}R(\mu)^if^{(i)}_K\in C([0,t_f];\ker R(\mu)^k)
\end{align}
and fulfills
\begin{align*}
R(\mu)x_2=-\sum_{i=0}^{k-1}R(\mu)^{i+1}f^{(i)}_K=-\sum_{i=1}^{k-1}R(\mu)^if^{(i-1)}_K\in C^1([0,t_f];\ker R(\mu)^k).
\end{align*}
Therefore, $x_2$ is a classical solution of the right-hand side equation in \eqref{decoupled}. Summarizing, we have shown that $x=x_1+x_2$ is a classical solution of \eqref{psresDGL_allg} for the initial value
\[
x(0)=x_1(0)+x_2(0)=x_0-\sum_{i=0}^{k-1}R(\mu)^if^{(i)}_K(0).
\]

If $x$ is a mild solution, then the unique solvability of~\eqref{psresDGL_allg} is a consequence of the unique solvability of the homogeneous equation obtained in Proposition~\ref{prop:homo}. To construct a mild solution of~\eqref{psresDGL_allg}, we derive the mild solutions of~\eqref{decoupled}. A mild solution of~\eqref{transformed_f_R} is given for all $f_R\in L^1([0,t_f];\overline{\ran R(\mu)^k})$ and all $x_0\in \overline{\ran R(\mu)^k}$ by~\eqref{VDK}, see \cite[Thm.~10.1.8]{JacoZwar12}.
Further, $x_2$ as in~\eqref{x_k_def} is also a mild solution of $ \tfrac{\rm d}{{\rm d} t}R(\mu)x_2(t)=x_2(t)+f_K(t)$ and the initial values satisfy~\eqref{IVal}.
 \end{proof}

\begin{remark}\
\begin{itemize}
    \item[\rm (i)] In~\cite{ReisTisc05} inhomogeneous ADAEs in Hilbert spaces where solved based on the Laplace transform. Here it was assumed that $\Cpo\subseteq \rho(E,A)$ for some $\omega\in\R$ and that the tractability index is finite, see Section~\ref{sec:indexsol}. Here, instead of \eqref{inhom}, the ADAE $E\tfrac{\rm d}{{\rm d}t}x(t)=Ax(t)+f(t)$, $t\in[0,t_f]$, was considered. Hence for the existence of solutions, the stronger assumption $f\in W^{k,1}([0,t_f];X)$ is required.

    \item[\rm (ii)] The differentiability assumption on $f_K$ can be relaxed if $R(\mu)^{i}f_K=0$ for some $i<k-1$. Then it is sufficient to assume $f_K\in W^{i,2}([0,t_f],\ker R(\mu)^k)$ for the existence of  a classical solution and $f_K\in W^{i-1,2}([0,t_f],\ker R(\mu)^k)$ for the existence of a mild (weak) solution.
\item[\rm (iii)]  If the pseudo-resolvent $R$ associated with an ADAE fulfills ($\mathbf{D_1}$), then \eqref{f_dec} is trivially fulfilled. Hence Proposition~\ref{prop:inhom} guarantees the existence of mild solution if $f_K\in L^2([0,t_f],\ker R( \mu))$. It follows from Proposition~\ref{prop:bnddiss} that the assumption ($\mathbf{D_1}$) is e.g.\ fulfilled for standard boundary control systems that are rewritten as ADAEs. In this particular case, the conditions for the existence of a mild solution is slightly weaker then weak differentiability assumptions on the inputs which were used in~\cite{JacoZwar12}.
\end{itemize}
\end{remark}

If $X$ is a Hilbert space and $\mathbf{(D_k)}$ holds, then we can use the space decomposition $X=\Vc_k\oplus\Wc_k\oplus\ldots\Wc_1$ and the staircase form \eqref{staircase} applied to the pseudo-resolvent to solve the pseudo-resolvent equation \eqref{psresDGL_allg} and hence the inhomogeneous ADAE \eqref{inhom}. 

\begin{proposition}
Let $X$ be a Hilbert space and $R:\Omega\rightarrow L(X)$ be a pseudo-resolvent which fulfills $\mathbf{(D_k)}$ for some $k\geq 1$ and let the associated operator $A_R$ be the generator of a strongly continuous semigroup and $f=\hat f_{\mu}:[0,t_f]\rightarrow X$ the inhomogeneity in \eqref{psresDGL_allg}.  Assume that $f_i:=P_{\mathcal{W}_i}f\in C^{k+1-i}([0,t_f],X)$ holds for all $i=1,\ldots,k$ and $f_{k+1}:=P_{\mathcal{V}_k}f\in C([0,t_f],X)$. Then for each initial value $x_{k+1}^0\in\mathcal{V}_k$ there exists a unique classical solutions of \eqref{psresDGL_allg} with $x_{k+1}(0)=x_0$. 
\end{proposition}
\begin{proof}
According to the space decomposition $X=\Vc_k\oplus\Wc_k\oplus\ldots\Wc_1$ we write $x=(x_{k+1},\ldots,x_{1})$, and abbreviate $R_{i+1,j}:=P_{\mathcal{W}_i}R(\mu)|_{\mathcal{W}_{j}}$,  $i,j=1,\ldots,k-1$ and $R_{k+1,j}:=P_{\mathcal{V}_k}R(\mu)|_{\mathcal{W}_{j}}$, $j=1,\ldots,k$.  Therefore,  \eqref{psresDGL_allg} can be rewritten using 
the staircase form \eqref{staircase} 
\begin{align}
\label{eq:staircase_proof}
\frac{d}{dt}\begin{smallbmatrix}
  R(\mu)|_{\mathcal{V}_k} & R_{k+1,k}& R_{k+1,k-1} &\ldots &R_{k+1,1} \\ 0&0& R_{k,k-1} &\ldots& R_{k,1}\\ &\ddots~~~~~&&\ddots&\\ & &~~~~~\ddots&& R_{2,1}\\& & & 0~~~~&0
\end{smallbmatrix}\begin{bmatrix}
    x_{k+1}\\ x_k\\ \vdots\\ x_{1}
\end{bmatrix}=\begin{bmatrix}
    x_{k+1}\\ x_k\\ \vdots\\ x_{1}
\end{bmatrix}+\begin{bmatrix}
    f_{k+1}\\ f_k\\ \vdots\\ f_{1}
\end{bmatrix}.
\end{align}
Step-by-step solution of \eqref{eq:staircase_proof} starting from the last row results in $x_{1}=-f_1$ and
\[
x_{i}=-f_i-\frac{\rm d}{{\rm d}t}\sum_{j=1}^{i-1} R_{i,j}x_j,\quad  i=2,\ldots,k,
\]
where the derivatives exist and are continuous since $f_i\in C^{k+1-i}([0,t_f],X)$. Furthermore, $\sum_{j=1}^k R_{k+1,j}x_j$ is continuously differentiable.

It remains to derive a classical solution $x_{k+1}$ for the first line in \eqref{eq:staircase_proof}. To this end we combine the fixed terms and rewrite the first line with continuous inhomogeneity $f_{k+1}+\tfrac{\rm d}{{\rm d} t}\sum_{j=1}^k R_{k+1,j}x_j$. Therefore we can apply Proposition \ref{prop:inhom}~(a) which yields that for all $x_{k+1}^0\in\mathcal{V}_k$ there exists a unique classical solution $x_{k+1}$ such that $x_{k+1}(0)=x_{k+1}^0$ holds. This finishes the proof.
\end{proof}

\section{Two applications}
\label{sec:appl}
\subsection{Heat-wave coupling}

The following system of PDEs from \cite{BattPauo16,SchwZwar14} describes a wave equation on the interval $(-1,0)$  coupled at $\xi=0$ with a heat equation on the interval $(0,1)$
\[
u_{tt}(t,\xi)=u_{\xi\xi}(t,\xi),~ \xi\in(-1,0),\quad
        w_t(t,\xi)=w_{\xi\xi}(t,\xi),~ \xi\in(0,1).
\]
Furthermore, we impose Dirichlet boundary conditions at both endpoints, interface conditions at $\xi=0$ and initial values
\begin{align*}
        w(t,1)&=0, & u_\xi(-1,t)&=0, &&\\ u_\xi(t,0)&=w_\xi(t,0), & u_t(t,0)&=w(t,0), &&\\
        w(0,\xi)&=w_0(\xi),& u_t(0,\xi)&=v_0(\xi), &u(0,\xi)&=u_0(\xi).
    \end{align*}
    We follow~\cite{SchwZwar14} and write this altogether as an ADAE, using the characteristic function $\chi_I$ on an interval $I$, in the following way on $X=Z=L^2([-1,1])\times L^2([-1,1])$
    \begin{align*}
    \frac{d}{dt}\underbrace{\begin{pmatrix}I&0\\0&\chi_{(-1,0)}I\end{pmatrix}}_{=:E}\begin{pmatrix}x_1(t,\xi)\\x_2(t,\xi)\end{pmatrix}=\underbrace{\begin{pmatrix}0&\partial_\xi\\\partial_\xi&-\chi_{(0,1)}\end{pmatrix}}_{=:A}\begin{pmatrix}x_1(t,\xi)\\x_2(t,\xi)\end{pmatrix},\quad (Ex)(0)=Ex_0 \\ x_1(t,\xi)=\begin{cases}u_t(\xi,t), & \text{if $\xi<0$,}\\ w(\xi,t), & \text{if $\xi>0$,}\end{cases},\quad x_2(\xi,t)=\begin{cases}u_\xi(\xi,t), & \text{if $\xi<0$,}\\ w_{\xi}(\xi,t), & \text{if $\xi>0$,}\end{cases}
    \end{align*}
    where $E$ is clearly bounded and $A$ is densely defined on $X$ with domain
    \begin{align*}
    \dom A&:=\left\{(x_1,x_2)\in X : x_1,x_2\in W^{1,2}([-1,1]), x_2(-1)=x_1(1)=0\right\},\\ x_0(\xi)&:=\begin{pmatrix}u_0(\xi)\\v_0(\xi)\end{pmatrix},\, \text{if $\xi<0$ and} \quad  x_0(\xi):=\begin{pmatrix}w_0(\xi)\\0\end{pmatrix},\, \text{if $\xi>0$.}
    \end{align*}
The invertibility of the off-diagonal entries $\partial_{\xi}$ on their respective domains implies that $A$ is also invertible and therefore  $0\in\rho(E,A)\neq\emptyset$. Furthermore, we have from integration by parts that for all $x=(x_1,x_2)\in\dom A$
\begin{align*}
 (Ax,x)_{L^2([-1,1])^2}&= \left(\begin{pmatrix}0&\partial_\xi\\\partial_\xi&-\chi_{(0,1)}\end{pmatrix}\begin{pmatrix}x_1\\x_2\end{pmatrix},\begin{pmatrix}x_1\\x_2\end{pmatrix}\right)_{L^2(-1,1)^2}\\&=\int_{-1}^1\partial_{\xi}x_2(\xi)\overline{x_1}(\xi)d\xi+\int_{-1}^1(\partial_{\xi}x_1(\xi)-\chi_{(0,1)}(\xi)x_2(\xi))\overline{x_2(\xi)}d\xi\\
&=\underbrace{[x_1(\xi)x_2(\xi)]_{\xi=-1}^1}_{=0}-\int_{-1}^1x_2(\xi)\overline{\partial_{\xi} x_1}(\xi)d\xi\\&~~~~+\int_{-1}^1(\partial_{\xi}x_1(\xi)-\chi_{(0,1)}(\xi)x_2(\xi))\overline{x_2(\xi)}d\xi.
\end{align*}
Therefore, we have for all $x=\begin{smallpmatrix}x_1\\x_2\end{smallpmatrix}\in\dom A$ that
\[
\re(Ax,x)_{L^2(-1,1)^2}=-\int_{-1}^1\chi_{(0,1)}(\xi)|x_2(\xi)|^2d\xi\leq 0
\]
which implies with Proposition~\ref{prop:suff_for_D2} that $\mathbf{(D_2)}$ holds for $R_{l}$. Furthermore, the condition~\eqref{y_impli} is fulfilled, which implies that $\begin{smallbmatrix}
EA^{-1}\\ I_n
\end{smallbmatrix}(\ran E)$ is the graph of a maximal dissipative operator which can be used to obtain the solutions of that ADAE. Indeed, to verify~\eqref{y_impli}, let $(x_1,x_2)\in\ker E$ then $x_1=0$ and $x_2\equiv 0$ on $(-1,0)$. If $A(0,x_2)\in\ran E$ then looking at the second component yields $x_2=0$ on $(0,1)$ and therefore $x=0$. Note that this condition can be used to obtain the solution of the heat-wave system using the staircase form \eqref{y_impli}.

\subsection{Distributed element circuits}
An electrical circuit with distributed resistance $R$, capacitance $C$ and inductivity $L$ on the spatial domain $[0,1]\subseteq\R$ can be described by the system of partial differential equations
\begin{align*}
L(x)\dot I(t,x)&=-R(x)I(t,x)-\frac{\partial V(t,x)}{\partial x},\nonumber \\
C(x)\dot V(t,x)&=-\frac{\partial I(t,x)}{\partial x}-G(x)V(t,x),\\ V(t,0)&=u_0(t),\, -I(1,t)=i_1(t)\nonumber
\end{align*}
where we assume that $L,C,R,G\in L^{\infty}([0,1])$ are nonnegative and that $L$ and $C$ are bounded away from zero on their support. This can be rewritten as an ADAE with $X=L^2([0,1])\times L^2([0,1])$, $Z=L^2([0,1])\times L^2([0,1])\times \R^2$ and
\[
E=\begin{pmatrix} L(\cdot)&0\\0&C(\cdot)\\0&0\\0&0\end{pmatrix},\quad A=\begin{pmatrix} -R(\cdot)&-\partial_{\xi}\\-\partial_{\xi} &-G(\cdot)\\\delta_1&0\\0&\delta_0\end{pmatrix}
\]
with $\dom A:=W^{1,2}([0,1])\times W^{1,2}([0,1])$ and $\delta_{x}$ is the point evaluation at $x\in[0,1]$. We show in the following that the resolvent $R_{l}$ fulfills $\mathbf{(D_2)}$ by applying Proposition~\ref{prop:suff_for_D2}.

First, note that \cite[Thm.~2.1]{SingKuma77} implies that $E$ has closed range. The dissipativity of $A$ can be shown analogously as for the heat-wave coupling.

Next, we show that $0\in\rho(A)$. To prove invertibility of $A$, we consider
 \[
 W^{1,2}_{x,0}([0,1]):=\{g\in W^{1,2}([0,1]) ~|~ g(x)=0\},~ x=0,1,\quad  \partial_{\xi,x}:=\partial_{\xi}\vert_{W^{1,2}_{x,0}([0,1])}.
 \]
 and first show the invertibility of
\begin{align*}
\hat A&:=\begin{bmatrix}-R(\cdot)\vert_{W^{1,2}_{1,0}([0,1])}&-\partial_{\xi,0}\\-\partial_{\xi,1}&-G(\cdot)\vert_{W^{1,2}_{1,0}([0,1])}\end{bmatrix}\\
\dom \hat A&:=W^{1,2}_{1,0}([0,1])\times W^{1,2}_{0,0}([0,1]).
\end{align*}
Clearly, interchanging the block columns of $\hat A$ preserves the invertibility and hence the invertibility of  $\hat A$ would follow from the invertibility of $\partial_{\xi,0}$ together with the invertibility of the Schur complement $\partial_{\xi,1}-G\partial_{\xi,0}^{-1}R$. This can easily be concluded from Fredholm theory, see \cite[p.~358]{GohbGold03}.
Indeed, $\partial_{\xi,1}-G\partial_{\xi,0}^{-1}R$ is a Fredholm operator with index zero, since $\partial_{\xi,1}$ is invertible and
$\partial_{\xi,0}^{-1}$ as well as $G\partial_{\xi,0}^{-1}R$ are  compact. To prove the invertibility of the Schur complement it is therefore enough to prove that it is injective. Assume conversely that there is an $g\in W^{1,2}_{1,0}([0,1])$ satisfying
 \[
 \partial_{\xi,1}g(x)-R\partial_{\xi,0}^{-1}Gg(x)=g'(x)-R(x)\int_0^xG(\tau)g(\tau)d\tau=0,\quad x\in[0,1].
 \]
 If we assume that $g(0)>0$ then $R,G\geq 0$ implies that the derivative is nonnegative for all $x\in[0,1]$. Invoking $g\in W^{1,2}_{1,0}([0,1])$, leads to $0=g(1)\geq g(0)=1$, the desired contradiction. The same conclusion holds for $g(0)<0$.
 Thus $g=0$, and hence $\partial_{\xi,1}-G\partial_{\xi,0}^{-1}R$ is invertible, meaning that $\hat A$ is also invertible.

 This together with the surjectivity of the point evaluations $\delta_0,\delta_1$, implies that $A$ is  surjective. Moreover, $A$ is injective, since $A(g_1,g_2)=0$ with $g_i\in W^{1,2}([0,1])$, $i=1,2$ implies $\delta_1 g_1=0$ and $\delta_0 g_2=0$. Hence $g=(g_1,g_2)\in W^{1,2}_{0,0}([0,1])\times W^{1,2}_{1,0}([0,1])$. The invertibility of $\hat A$ implies $g=0$.

\section*{Acknowledgments}
This work was supported by the grants RE 2917/4-1 and WO 2056/1-1 ``Systems theory of partial differential-algebraic equations'' by the Deutsche Forschungsgemeinschaft (DFG). The authors would like to thank Jochen Gl\"{u}ck and Felix L.\ Schwenninger for valuable discussions and suggestions.

\bibliographystyle{plain}
\bibliography{references}

\begin{thebibliography}{10}

\bibitem{Achl22}
F.~Achleitner, A.~Arnold, and V.~Mehrmann.
\newblock Hypocoercivity and hypocontractivity concepts for linear dynamical
  systems.
\newblock {\em The Electronic Journal of Linear Algebra}, 39, 2022.

\bibitem{Alt16}
H.W. Alt.
\newblock {\em Linear Functional Analysis, An Application-Oriented
  Introduction}.
\newblock Universitext. Springer London, 2016.

\bibitem{Aren87}
W.~Arendt.
\newblock Vector-valued {L}aplace transforms and {C}auchy problems.
\newblock {\em Israel J. Math}, 59:327--352, 1987.

\bibitem{Aren01}
W.~Arendt.
\newblock Approximation of degenerate semigroups.
\newblock {\em Taiwanese J. Math}, 5(2):279--295, 2001.

\bibitem{ArenBatt11}
W.~Arendt, C.J.K. Batty, M.~Hieber, and F.~Neubrander.
\newblock {\em Vector-valued {L}aplace Transforms and {C}auchy Problems, 2nd
  edition}, volume~96 of {\em Monographs in Mathematics}.
\newblock Birkh\"{a}user, Basel, 2011.

\bibitem{BasaCher02}
A.~Baskakov and K.~Chernishov.
\newblock Spectral analysis of linear relations and degenerate operator
  semi-groups.
\newblock {\em Sb.\ Math.}, 193(11):1573--1610, 2002.

\bibitem{BattPauo16}
C.J.K. Batty, L.~Pauonen, and D.~Seifert.
\newblock Optimal energy decay in a one-dimensional coupled wave-heat system.
\newblock {\em Journal of Evolution equations}, 16:649--664, 2016.

\bibitem{BdGCS18}
S.~Baumanns, H.~De Gersem, I.~Cortes Garcia, and S.~Sch\"ops.
\newblock Systems of differential algebraic equations in computational
  electromagnetics.
\newblock In A.~Ilchmann and T.~Reis, editors, {\em Applications of
  Differential-Algebraic Equations: Examples and Benchmarks},
  Differential-Algebraic Equations Forum, pages 123--169. Springer, Berlin
  Heidelberg, 2019.

\bibitem{BehrHassdeSn20}
J.~Behrndt, S.~Hassi, and H.~de~Snoo.
\newblock {\em Boundary Value Problems, Weyl Functions, and Differential
  Operators}, volume 108 of {\em Monographs in Mathematics}.
\newblock Birkh\"{a}user, Basel, 2020.

\bibitem{BergIlch12a}
T.~Berger, A.~Ilchmann, and S.~Trenn.
\newblock The quasi-{W}eierstra{\ss} form for regular matrix pencils.
\newblock {\em Linear Algebra Appl.}, 436(10):4052--4069, 2012.

\bibitem{BergReis14b}
T.~Berger and T.~Reis.
\newblock Structural properties of positive real and reciprocal rational
  matrices.
\newblock In {\em Proceedings of the MTNS 2014}, pages 402--409, Groningen, NL,
  2014.

\bibitem{Camp82}
S.L. Campbell.
\newblock {\em Singular Systems of Differential Equations I}, volume~2.
\newblock Pitman Publishing, 1982.

\bibitem{Show74}
R.W. Carroll and R.E. Showalter.
\newblock {\em Singular and Degenerate Cauchy Problems}.
\newblock ISSN. Elsevier Science, 1977.

\bibitem{Cross}
R.~Cross.
\newblock {\em Multivalued {L}inear {O}perators}.
\newblock Marcel Dekker, New York, 1998.

\bibitem{Diestel77}
J.~Diestel and J.J. Uhl.
\newblock {\em Vector Measures}, volume~15 of {\em Mathematical surveys and
  monographs}.
\newblock American Mathematical Society, Providence, RI, 1977.

\bibitem{EngeNage06}
K.-J. Engel and R.~Nagel.
\newblock {\em A Short Course on Operator Semigroups}.
\newblock Springer Science \& Business Media, 2006.

\bibitem{FaviYagi99}
A.~Favini and A.~Yagi.
\newblock {\em Degenerate Differential Equations on Banach Spaces}.
\newblock Marcel Dekker, New York, 1999.

\bibitem{Gant59a}
F.~R. Gantmacher.
\newblock {\em The Theory of Matrices (Vol.~I)}.
\newblock Chelsea Pulications, New York, 1959.

\bibitem{GernHall20}
H.~Gernandt, F.E. Haller, and T.~Reis.
\newblock A linear relations approach to port-{H}amiltonian
  differential-algebraic equations.
\newblock {\em SIAM Journal Matrix Anal.\ Appl.}, 2021.
\newblock in press.

\bibitem{GernMoal20}
H.~Gernandt, N.~Moalla, F.~Philipp, W.~Selmi, and C.~Trunk.
\newblock Invariance of the essential spectra of operator pencils.
\newblock {\em Oper. Theory Adv. Appl.}, 278:203--219, 2020.

\bibitem{GohbGold03}
I.~Gohberg, S.~Goldberg, and M.~Kaashoek.
\newblock {\em Basic Classes of Linear Operators}.
\newblock Birkh{\"a}user, 2003.

\bibitem{Haas06}
M.~Haase.
\newblock {\em The Functional Calculus for Sectorial Operators}, volume 169 of
  {\em Operator Theory: Advances and Applications}.
\newblock Birh\"{a}user, Basel, Switzerlandl, 2006.

\bibitem{Hill48}
E.~Hille.
\newblock {\em Functional Analysis and Semigroups}.
\newblock Amer. Math. Soc., Providence, RI, 1948.

\bibitem{HillPhil96}
E.~Hille and R.S. Phillips.
\newblock {\em Functional Analysis and Semi-groups}.
\newblock American Mathematical Soc., 1996.

\bibitem{JacoZwar12}
B.~Jacob and H.~Zwart.
\newblock {\em Linear Port-Hamiltonian Systems on Infinite-dimensional Spaces},
  volume 223 of {\em Operator Theory: Advances and Applications}.
\newblock Birkh\"{a}user, 2012.

\bibitem{JacoMorr22}
Birgit Jacob and Kirsten Morris.
\newblock On solvability of dissipative partial differential-algebraic
  equations.
\newblock {\em IEEE Control Syst.\ Lett.}, 6:3188--3193, 2022.

\bibitem{Kato59}
T.~Kato.
\newblock Remarks on pseudo-resolvents and infinitesimal generators of
  semi-groups.
\newblock {\em Proceedings of the Japan Academy}, 35(8):467--468, 1959.

\bibitem{Kato80}
T.~Kato.
\newblock {\em Perturbation Theory for Linear Operators}.
\newblock Springer, Berlin Heidelberg, Germany, 2nd edition, 1980.

\bibitem{KunkMehr06}
P.~Kunkel and V.~Mehrmann.
\newblock {\em Differential-Algebraic Equations. Analysis and Numerical
  Solution}.
\newblock EMS Publishing House, Z{\"u}rich, Switzerland, 2006.

\bibitem{LamoMarz13}
Ren\'{e} Lamour, Roswitha M{\"a}rz, and Caren Tischendorf.
\newblock {\em Differential Algebraic Equations: A Projector Based Analysis},
  volume~1 of {\em Differential-Algebraic Equations Forum}.
\newblock Springer-Verlag, Heidelberg-Berlin, 2013.

\bibitem{Marz96}
R.~M{\"a}rz.
\newblock Canonical projectors for linear differential algebraic equations.
\newblock {\em Computers Math. Appl.}, 31(4/5):121--135, 1996.

\bibitem{MehlMehrWojt18}
C.~Mehl, V.~Mehrmann, and M.~Wojtylak.
\newblock Linear algebra properties of dissipative {H}amiltonian descriptor
  systems.
\newblock {\em SIAM Journal Matrix Anal.\ Appl.}, 39:1489--1519, 2018.

\bibitem{MehrZwar23}
V.~Mehrmann and H.~Zwart.
\newblock Abstract dissipative {H}amiltonian differential-algebraic equations
  are everywhere.
\newblock {\em arXiv:2311.03091}, 2023.

\bibitem{Pazy83}
A.~Pazy.
\newblock {\em Semigroups of Linear Operators and Applications to Partial
  Differential Equations}, volume~44 of {\em Applied Mathematical Sciences}.
\newblock Springer Science \& Business Media, 1983.

\bibitem{Reis06}
T.~Reis.
\newblock {\em Systems Theoretic Aspects of {PDAE}s and Applications to
  Electrical Circuits}.
\newblock Doctoral dissertation, Fachbereich Mathematik, Technische
  Universit{\"a}t Kaiserslautern, Kaiserslautern, 2006.

\bibitem{Reis07}
T.~Reis.
\newblock Consistent initialization and perturbation analysis for abstract
  differential-algebraic equations.
\newblock {\em Math. Control Signals Systems}, 19(3):255--281, 2007.

\bibitem{ReisTisc05}
T.~Reis and C.~Tischendorf.
\newblock Frequency domain methods and decoupling of linear infinite
  dimensional differential algebraic systems.
\newblock {\em Journal of Evolution equations}, 5(3):357--385, 2005.

\bibitem{Rudi91}
W.~Rudin.
\newblock {\em Functional Analysis}.
\newblock McGraw-Hill Science, Engineering \& Mathematics, 1991.

\bibitem{SchwZwar14}
F.L. Schwenninger and H.~Zwart.
\newblock Generators with a closure relation.
\newblock {\em Operators and Matrices}, 8(1):157--165, 2014.

\bibitem{Sime13}
B.~Simeon.
\newblock {\em Computational Flexible Multibody Dynamics}.
\newblock Differential-Algebraic Equations Forum. Springer, Heidelberg-Berlin,
  2013.

\bibitem{SingKuma77}
R.~K. Singh and A.~Kumar.
\newblock Multiplication operators and composition operators with closed
  ranges.
\newblock {\em Bulletin of The Australian Mathematical Society}, 16:247--252,
  1977.

\bibitem{SvirFedo03}
G.~A. Sviridyuk and V.~E. Fedorov.
\newblock {\em Linear Sobolev type equations and degenerate semigroups of
  operators}.
\newblock Inverse and ill-posed problems Series. De Gruyter, Utrecht, 2003.

\bibitem{ThalThal99}
B.~Thaller and S.~Thaller.
\newblock Semigroup theory of degenerate linear {C}auchy problems.
\newblock Technical report, Universit\"{a}t Graz/Technische Universit\"{a}t
  Graz. SFB F003-Optimierung und Kontrolle, 1999.

\bibitem{Tis03}
C.~Tischendorf.
\newblock Modeling, simulation, and optimization of integrated circuits.
\newblock In K.~Antreich, R.~Bulirsch, Gilg A., and P.~Rentrop, editors, {\em
  Modeling circuit systems coupled with distributed semiconductor equations},
  ISNM International Series of Numerical Mathematics, pages 229--247.
  Birh\"{a}user, Basel, Switzerlandl, 2003.

\bibitem{Tros19}
S.~Trostorff.
\newblock Semigroups associated with differential-algebraic equations.
\newblock In {\em Semigroups of Operators-- Theory and Applications}, pages
  79--94. Springer, 2020.

\bibitem{TrosWaur18}
S.~Trostorff and M.~Waurick.
\newblock On higher index differential-algebraic equations in infinite
  dimensions.
\newblock In A.~B\"{o}ttcher, D.~Potts, P.~Stollmann, and D.~Wenzel, editors,
  {\em The Diversity and Beauty of Applied Operator Theory}, Operator Theory
  Advances Appl., pages 477--486. Birh\"{a}user, Basel, Switzerland, 2018.

\bibitem{TrosWaur19}
S.~Trostorff and M.~Waurick.
\newblock On differential-algebraic equations in infinite dimensions.
\newblock {\em Journal of Differential Equations}, 266(1):526--561, 2019.

\end{thebibliography}

\end{document}